\documentclass[12pt,final]{amsart}
\usepackage{xr,xr-hyper}
\usepackage{amssymb,todonotes,hyperref,amsfonts,mathrsfs,tikz,euscript,cancel,nicefrac}
\title{On uniqueness of  tensor products of irreducible categorifications}
\author{Ivan Losev} \author{Ben Webster}
\renewcommand{\sl}{\mathfrak{sl}}
\newcommand{\Cat}{\mathcal{C}}
\newcommand{\OCat}{\mathcal{O}}

\newcommand{\C}{\mathbb{K}}

\newcommand{\la}{\lambda}
\newcommand{\al}{\alpha}
\newcommand{\Z}{\mathbb{Z}}

\newcommand{\Bb}{\mathbf{b}}
\newcommand{\cata}{\mathfrak{V}}
\newcommand{\Ba}{\mathbf{a}}

\newcommand{\Bi}{\mathbf{i}}
\newcommand{\bmu}{{\boldsymbol{\underline{\mu}}}}
\newcommand{\bnu}{{\boldsymbol{\underline{\nu}}}}
\newcommand{\bla}{{\boldsymbol{\underline{\lambda}}}}

\newcommand{\End}{\operatorname{End}}
\newcommand{\excise}[1]{}
\newcommand{\arxiv}[1]{\href{http://arxiv.org/abs/#1}{\tt arXiv:\nolinkurl{#1}}}
\newcommand{\cO}{\mathcal{O}}
\newcommand{\g}{\mathfrak{g}}
\newcommand{\fg}{\mathfrak{g}}

\newcommand{\Hom}{\operatorname{Hom}}

\newcommand{\Ext}{\operatorname{Ext}}

\newcommand{\Efun}{\mathcal{E}}
\newcommand{\Ffun}{\mathcal{F}}
\newcommand{\gl}{\mathfrak{gl}}

\newcommand{\gr}{\operatorname{gr}}

\newcommand{\mmod}{\operatorname{-mod}}

\newcommand{\oDelta}{\bar{\Delta}}
\newcommand{\onabla}{\bar{\nabla}}
\unitlength=1mm   \numberwithin{equation}{section}
\newtheorem{Thm}{Theorem}[section]
\newtheorem{Prop}[Thm]{Proposition}
\newtheorem{Cor}[Thm]{Corollary}
\newtheorem{Lem}[Thm]{Lemma}
\newtheorem{itheorem}{Theorem}

\theoremstyle{definition}

\newtheorem{defi}[Thm]{Definition}
\newtheorem{Rem}[Thm]{Remark}

\newcommand{\ivantodo}{\todo[inline,color=green!20]}
\newcommand{\te}{\tilde{e}}
\newcommand{\tf}{\tilde{f}}

\numberwithin{table}{section} \oddsidemargin=0cm
\address{I.L.: Department
of Mathematics, Northeastern University, Boston MA 02115 USA}
\email{i.loseu@neu.edu}
\address{B.W.: Department of Mathematics, University of Virginia,
Charlottesville, VA 22904 USA}
\email{btw4e@virginia.edu}
\begin{document}

\begin{abstract}
  In this paper, we propose an axiomatic definition for a {\bf tensor
    product categorification}. A tensor product categorification is an abelian category with a
  categorical action of a Kac-Moody algebra $\fg$ in the sense of
  Rouquier or Khovanov-Lauda whose Grothendieck group is isomorphic to
  a tensor product of simple modules.  However, we require a much
  stronger structure than a mere isomorphism of representations; most
  importantly, each such categorical representation must have a
  standardly stratified structure compatible with the categorification
  functors, and with combinatorics matching those of the tensor product.

 With these stronger conditions, we recover a uniqueness theorem
 similar in flavor to that of Rouquier for categorifications of simple
 modules.  Furthermore, we already know of an example of such a
 categorification: the representation category of an algebra $T^\bla$ previously
 defined by the second author using generators and relations.
 Next, we show that tensor product categorifications give a
 categorical realization of tensor product crystals analogous to that for
 simple crystals given by cyclotomic quotients of KLR algebras.

Examples of such categories are also readily found in more classical
representation theory; for finite and affine type A, tensor product
categorifications can be realized as quotients of the representation
categories of cyclotomic $q$-Schur algebras.
%
\end{abstract}
\maketitle

\section{Introduction}
\label{sec:introduction}

A subject that has attracted great attention in recent years is that
of categorical representations of a Kac-Moody Lie algebra $\fg$; this is the study of
2-categories corresponding to universal enveloping algebras of Lie
algebras, and in particular, their actions on categories.  This theory
has deep roots, but the notion of a categorical action of
$\mathfrak{sl}_2$ was first introduced by Chuang and Rouquier
\cite{CR04}, and broadened to other Kac-Moody algebras and developed
further by Khovanov and Lauda
\cite{Lausl21,KLIII} and Rouquier \cite{Rou2KM}.

Obviously,
one important question is the relationship between categorical
representations and the usual linear representations of $\fg$.
For simple representations, this relationship is quite direct: each
simple linear representation of $\fg$ has a universal
categorification, by work of Rouquier \cite{Rou2KM}. In essence,
simple linear representations and simple categorical representations
are in bijection.

However,
categorical representations are not necessarily semi-simple, even if
the representation on their Grothendieck group is.  Examples show that
interesting non-irreducible representations should have
categorifications which are ``more than the sum of their parts''; in this paper, our main
example is tensor products of irreducibles, but the same principle
applies to the categorifications of Fock space supplied by the
categories $\mathcal{O}$ over Cherednik algebras \cite{ShanCrystal,WebRou}.

In particular, the second author defined a category $\cata^\bnu$
attached to a list of highest weights $\bnu=(\nu_1,\dots, \nu_\ell)$
in \cite{Webmerged}; in this paper, we will denote this category
$\mathcal{C}(\bnu)$.  This carries a categorical action of $\fg$ and has
Grothendieck group isomorphic to the tensor product of $U(\fg)$
modules $V_{\nu_1}\otimes \cdots \otimes V_{\nu_\ell}$ (or the
corresponding representations of $U_q(\fg)$ if one incorporates the
grading).  Many properties of this category suggest that it is the
``right'' categorification for tensor products; in particular, the
ribbon structure on the category of $U_q(\fg)$ modules, and the
canonical basis of Lusztig both have appropriate categorical
analogues.  However, the definition given in \cite{Webmerged} is {\it ad
  hoc}, defined by generators and relations, and lacks a universal
property.  In this paper, we try to correct this defect, giving an
axiomatic characterization of this category, suggested by the notion
of a {\bf highest weight categorification} introduced by the first
author \cite{LoHWCI,LoHWCII} (see, in particular, the discussion of basic
highest weight $\sl_m$-categorifications  in \cite[9.2]{LoHWCII}).

In the crudest sense, the category $\Cat(\bnu)$ is the unique abelian
category which carries a categorical $\fg$-action and which is
obtained by beginning with the unique categorification of
$V_{\nu_1}\otimes \cdots \otimes V_{\nu_\ell}$ as an irreducible
$\fg^{\oplus \ell}$-module, and then adding in new extensions between
the projectives in this category in a controlled way.  We require that the resulting category is {\it
  standardly stratified}, that is, it has a subcategory of standardly
filtered modules which is closed under categorification functors, and
whose combinatorics are controlled by those of the tensor product (in
particular, the preorder used  in the stratification depends on
the order of the tensor factors).  We formalize this idea with the
definition of a {\bf tensor product categorification} (\S \ref{SS_definition}).  The main body
of the paper is dedicated to the proof of the uniqueness result
described above:
\begin{itheorem}[Thm.~\ref{main2}, Thm.~\ref{Thm:cryst}]\label{main}
  Any tensor product categorification for the representation
  $V_{\nu_1}\otimes \cdots \otimes V_{\nu_\ell}$ is strongly
  equivariantly equivalent to $\Cat(\bnu)$ as a categorical
  $\fg$-module.

  For any such categorification, there is a canonical isomorphism of
  crystals between the isomorphism classes of simple objects and the
  tensor product crystal $B(\nu_1)\otimes \cdots \otimes B(\nu_\ell)$
  (as  conjectured in \cite[\S \ref{m-sec:crystal}]{Webmerged}).
\end{itheorem}
\excise{
We note a number of consequences of this observation which are not
obvious from the definition of a tensor product category; in
particular, in type A, all such categorifications are quotients of
categories of representations over cyclotomic $q$-Schur algebras
(Corollary \ref{Cor:type_A_fund}).}

The proof is by induction; we use the notion of {\bf categorical
  splitting}, introduced by the first author in \cite{LoHWCII}.
Roughly, inside any categorification of $V_{\nu_1}\otimes \cdots
\otimes V_{\nu_\ell}$, one finds a categorification of $V_{\nu_1}\otimes
\cdots \otimes V_{\nu_{\ell-1}}$. The inductive hypothesis allows us
to identify this subcategory with
$\Cat(\nu_1,\dots,\nu_{\ell-1})$; we can then argue that any tensor product  $\fg$-categorification
containing an appropriately embedded copy of
$\Cat(\nu_1,\dots,\nu_{\ell-1})$ must be $\Cat(\bnu)$.

This theorem has quite powerful applications in the theory of Lie
superalgebras, which are
explored further in a joint paper of Brundan and the
authors \cite{BLW}.

We should note that this theorem is likely not the last word in the
question of how one can categorify a tensor product.  Of particular
import is work announced by Rouquier, which proposes a notion of
internal tensor product for the 2-category of categorical
$\fg$-actions.  Obviously, we anticipate that our tensor product
categorifications are the categories we would arrive at using
Rouquier's internal tensor product, but this remains to be confirmed.

\section*{Acknowledgements}

I.L. was supported by the NSF under Grant DMS-1161584.
B.W. was supported by the NSF under Grant DMS-1151473.

\section{Standardly stratified categories}
\subsection{Definition}
Let $\C$ be a field of any characteristic. We consider an  abelian
category $\Cat$ such that each block of $\Cat$ is equivalent to the category of
finite dimensional modules over a finite dimensional $\C$-algebra, and
such that, for every simple object $L$, we have $\End_{\Cat}(L)\cong\C$, i.e. every
irreducible in $\Cat$ is absolutely irreducible.
We will consider a set $\Lambda$ with fixed bijection to
the isomorphism classes of simple objects in $\Cat$.
For any $\la\in \Lambda$, we let $L(\lambda)$ denote the corresponding simple and let $P(\lambda)$ be its
projective cover.

Now consider a poset $\Xi$ with the sets $\{\xi'\in \Xi\vert \xi'>\xi\}$
and $\{\xi'\in \Xi\vert \xi'<\xi\}$ finite for each $\xi\in \Xi$.
Choose a map $\varrho:\Lambda\rightarrow \Xi$
with finite fibers; this map induces a natural preorder on  $\Lambda$. To each $\xi\in \Xi$, we assign the
Serre subcategories $\Cat_{\leqslant \xi}$ (resp., $\Cat_{<\xi}$)
of $\Cat$ spanned by $L(\lambda)$ with $\varrho(\lambda)\leqslant \xi$ (resp., $\varrho(\lambda)<\xi$).
Of course, if $\xi\leqslant \xi'$, then $\Cat_{\leqslant \xi}\subset
  \Cat_{\leqslant \xi'}$. For $\xi\in \Xi$, set
$\Cat_\xi:=\Cat_{\leqslant \xi}/\Cat_{<\xi}$. For $\lambda\in \varrho^{-1}(\xi)$ let $L_\xi(\lambda)$ denote the simple
object in $\Cat_\xi$ corresponding to $\lambda$. Let $P_\xi(\lambda)$ denote the projective cover
of $L_\xi(\lambda)$ in $\Cat_\xi$. 

Let $\pi_\xi$ denote the quotient functor $\Cat_{\leqslant \xi}\twoheadrightarrow \Cat_\xi$.
We suppose that this functor has an exact left   adjoint functor.
\begin{defi}\label{defi:ststr}
  We will call this left adjoint the {\bf standardization} functor and
  denote it by $\Delta_\xi$.  We will often omit $\xi$ from the
  notation. For $\lambda\in \Lambda$ let $\Delta(\lambda)$ (resp.,
  $\oDelta(\lambda)$) denote the object
  $\Delta_{\varrho(\lambda)}(P_\xi(\lambda))$ (resp.,
  $\Delta_{\varrho(\lambda)}(L_\xi(\lambda))$). The objects
  $\Delta(\lambda),\oDelta(\lambda)$ will be called {\bf standard} and
  {\bf proper standard}.

We call the category $\Cat$ equipped with a filtration $\Cat_{\leqslant \xi}$ (such that $\Delta_\xi$ is an exact functor) a {\bf standardly stratified}
category if there is an epimorphism $P(\lambda)\twoheadrightarrow \Delta(\lambda)$ whose kernel admits a filtration
by objects $\Delta(\mu)$ with $\mu>\lambda$.
\end{defi}
If for each $\xi$, the category $\Cat_\xi$ coincides with the category $\operatorname{Vect}$ of vector spaces,
then we arrive at the usual definition of a highest weight category.

\begin{Rem}
  We would like to point out that one definition of a standardly
  stratified category already exists in the literature,
  \cite[2.2.1]{CPS96}. A standardly stratified category in our sense
  is also standardly stratified in the sense of \cite{CPS96}, but our
  definition is more restrictive: in \cite{CPS96} it is not required
  that $\pi_\xi$ admits an exact left adjoint. We feel that our
  definition is more natural, and the structure theory behaves better
  (for example, see Lemma \ref{Lem:co-st-ext} below). So, despite the fact that our
  axioms differ from \cite{CPS96}, we will still call categories from Definition \ref{defi:ststr} ``standardly stratified.''
\end{Rem}

We remark that, by the definition of $\Delta(\lambda), \oDelta(\lambda)$,
there is an epimorphism $\Delta(\lambda)\twoheadrightarrow \oDelta(\lambda)$ and the head of $\Delta(\lambda)$
is simple and coincides with $L(\lambda)$. We also remark that the simple constituents of
the radical of $\oDelta(\lambda)$ are of the form $L(\mu)$ with $\mu<\lambda$.


It is a standard fact that the condition on a filtration of projectives implies
\begin{align}
&\Ext_{\Cat}^i(\Delta_\xi(M),\Delta_{\xi'}(M'))=0, \text{ for
}\xi\not\leqslant\xi', i>0. \label{eq:Ext_vanish}\\
&\Ext_{\Cat}^i(\Delta_\xi(M),\Delta_{\xi}(M'))=\Ext^i_{\Cat_\xi}(M,M'), i\geqslant 0. \label{eq:Ext_preserv}
\end{align}

Let $\Cat^{\Delta}$ (resp., $\Cat^{\oDelta}$)
denote the full subcategory of $\Cat$ consisting of all objects admitting a filtration whose successive
quotients are standard (resp., proper standard) objects. So, in particular,
$\Cat-\operatorname{proj}\subset \Cat^{\Delta}\subset \Cat^{\oDelta}$.
The following lemma is a direct corollary of (\ref{eq:Ext_vanish}).

\begin{Lem}\label{Lem:proj_exact}
Let $\iota_\xi$ be the inclusion functor $\Cat_{\leqslant \xi}\hookrightarrow \Cat$ and
$\iota_\xi^!$ be its left adjoint functor $\Cat\rightarrow \Cat_{\leqslant \xi}$. Then the functor
$\iota_\xi^!$ is exact on $\Cat^{\oDelta}$ (meaning that it maps exact sequences to exact sequences).
\end{Lem}

For any standardly stratified category $\Cat$, we can consider its
{\bf associated graded} \linebreak $\gr \Cat=\oplus_\xi\Cat_\xi$; we can view $\Delta:=\bigoplus_{\xi\in \Xi}\Delta_\xi$
as a faithful inclusion $\gr \Cat\to \Cat$, which fails to be full.

For an ideal $\Xi_0$ in $\Xi$ (i.e., $\xi\in \Xi_0, \xi'<\xi$ implies $\xi'\in \Xi_0$),
we can consider the Serre subcategory $\Cat_{\Xi_0}\subset \Cat$ spanned by $L(\lambda)$
with $\varrho(\lambda)\in \Xi_0$ and the quotient
$\Cat/\Cat_{\Xi_0}$. Both these categories have a natural standardly stratified
structure. Indeed, it is clear that the functors $\pi_\xi^!$ are exact for both these
categories and then we can use \cite[Theorem 2.2.6]{CPS96}.

The quotient functor
$\pi_{\Xi_0}:\Cat\twoheadrightarrow\Cat/\Cat_{\Xi_0}$  has a left adjoint $\pi_{\Xi_0}^!$ which
is exact on
$(\Cat/\Cat_{\Xi_0})^{\oDelta}$ (see the discussion of recollement in
\cite[\S 1.3]{CPS96} or \cite[4.1]{LoHWCII} for the case of highest weight categories, the general
case is completely analogous).
For each $\xi\in \Xi\setminus
\Xi_0$, this functor
satisfies $\Delta_\xi=\pi_{\Xi_0}^!\circ \Delta_{0,\xi}$, where $\Delta_{0,\xi}$ is
the standardization functor for $\Cat/\Cat_{\Xi_0}$. This is an easy corollary
of the triangularity property for the projectives and is shown as in the
proof of  \cite[3.5.1]{CPS96} or as in \cite[4.1]{LoHWCII}.
In particular, \[\pi_{\Xi_0}^!(\Delta_0(\lambda))=\Delta(\lambda)\qquad \pi_{\Xi_0}^!(\oDelta_0(\lambda))=\oDelta(\lambda).\]

\excise{We also remark that $\Cat_\xi$ is the kernel of  $\Cat/\Cat_{\Xi_0}\twoheadrightarrow \Cat/\Cat_{\Xi_0\sqcup\{\xi\}}$
for $\Xi_0=\{\xi'\in \Xi| \xi'\not\geqslant \xi\}$.}

\subsection{Costandard objects}
The category $\Cat^{opp}$ has a ``dual'' standardly stratified
structure, which we describe here.
Let $I(\lambda)$ denote the injective hull of $L(\lambda)$.
\begin{defi}
  Define the {\bf costandard} object $\nabla(\lambda)$ as the maximal
  subobject of $I(\lambda)$ whose simple constituents are of the form
  $L(\mu)$ with $\mu\leqslant \lambda$. Also define the proper costandard
  object $\onabla(\lambda)$ as the maximal subobject of $I(\lambda)$
  such that the simple subquotients of $\onabla(\lambda)/L(\lambda)$
  are of the form $L(\mu)$ with $\mu<\lambda$.

Let  $\Cat^\nabla$ (resp., $\Cat^{\onabla}$)
denote the full subcategories of $\Cat$ consisting of all objects admitting a filtration whose successive
quotients are  costandard (resp., proper costandard) objects.
\end{defi}

We remark that, by the definition, $\nabla(\lambda)$ is injective in $\Cat_{\leqslant \varrho(\lambda)}$.

\begin{Lem}\label{Lem:co-st-ext}
\label{Lem:cost_filt} \mbox{}
\begin{enumerate}
\item We have $\dim \Ext^i(\Delta(\lambda), \onabla(\mu))=\dim
  \Ext^i(\oDelta(\lambda), \nabla(\mu))=
  \delta_{i,0}\delta_{\lambda,\mu}$. 

 \item For $N\in \Cat$, we have $N\in \Cat^\nabla$ (resp., $N\in
  \Cat^{\onabla}$) if and only if $\Ext^1(\oDelta(\lambda),N)=0$
  (resp., $\Ext^1(\Delta(\lambda), N)=0$) for all $\lambda$.

\item The right adjoint functor $\nabla_\xi$ of the projection
$\Cat_{\leqslant \xi}\rightarrow \Cat_\xi$ is exact and  \[\nabla_\xi(I_\xi(\lambda))=\nabla(\lambda)\qquad \nabla_\xi(L_\xi(\lambda))=\onabla(\lambda).\]
\end{enumerate}

\end{Lem}
Here we write $I_\xi(\lambda)$ for the injective envelope of $L_\xi(\lambda)$ in $\Cat_\xi$.
We remark that here it is essential that we require the standardization functor to be exact,
see \cite[2.2.5]{CPS96} for a counter-example.
\begin{proof}
We start with (1).

 a) We prove that $\dim \Ext^i(\Delta(\lambda), \nabla(\mu))\neq 0$ implies $i=0$ and $\lambda,\mu$ are comparable. Let us deal with $i=0,1$. If $\lambda<\mu$, these cases follow from $\nabla(\mu)$ being injective in  $\Cat_{\leqslant \varrho(\mu)}$. If $\lambda\not\leqslant \mu$, then we use a filtration on $P(\lambda)$ and induction. Now that we know that $\Ext^1$ vanishes, we use filtrations on projectives to deduce the vanishing of higher Ext's.

b) To prove $\dim \Ext^i(\bar{\Delta}(\lambda), \nabla(\mu))=\delta_{i0}\delta_{\lambda\mu}$,  we use the exactness of $\Delta$ and look at a resolution of $\bar{\Delta}(\lambda)$ by means of $\Delta(\lambda')$ with $\lambda'$ comparable to $\lambda$. To have $\Ext^i(\bar{\Delta}(\lambda), \nabla(\mu))\neq 0$, we need $\mu$ to be comparable to $\lambda$. Then we can project to $\Cat_{\varrho(\lambda)}$ and use that the projection of $\nabla(\mu)$ is an indecomposable injective.

c) To prove $\dim \Ext^i(\Delta(\lambda), \bar{\nabla}(\mu))=\delta_{i0}\delta_{\lambda\mu}$ we use the definition of $\bar{\nabla}(\mu)$ to treat the $i=0$ case and the case of $i=1$ and $\lambda\leqslant \mu$ (note that
the extension with $\Delta(\lambda)$ on top and $\bar{\nabla}(\mu)$ on the bottom admits a morphism
to the indecomposable injective $I(\mu)$). The remaining cases are treated as in a).

Let us  turn to part (2);
furthermore, since the proofs are parallel, we only check that
$\Ext^1(\oDelta(\lambda),N)=0$ if and only if $N\in \Cat^\nabla$.
By induction, we may assume that $\Cat=\Cat_{\leqslant \xi}$  and that for objects in
$\Cat_{<\xi}$,  our claims are proved.

Let $N_0$ be the largest subobject in $N$ belonging to
$\Cat_{<\xi}$. Then for any $\lambda$ with $\varrho(\lambda)<\xi$ we
have $\Ext^1(\oDelta(\lambda),N_0)\cong \Hom(\oDelta(\lambda),N/N_0)=0$. Therefore, by the inductive assumption,  $N_0$
is $\nabla$-filtered
and hence $\Ext^i(\oDelta(\lambda),N_0)=0$ for all $i>0$. It follows that
$\Ext^1(\oDelta(\lambda),N/N_0)\cong \Ext^2(\oDelta(\lambda),N_0)=0$.
So it is enough to consider the case when $N_0=0$, i.e., the socle of
$N$ is a sum of
simples of the form $L(\mu)$ for  $\varrho(\mu)=\xi$. It follows that $N$ embeds to the sum  $I$ of several
$I(\mu)$'s with $\varrho(\mu)=\xi$ such that the socles of $I$ and of $N$ coincide.
But for such $\mu$ we have $I(\mu)=\nabla(\mu)$.
If we have $\Ext^1(\oDelta(\lambda), N)=0$ for all $\lambda$, then we
have a surjection $\Hom(\oDelta(\lambda),I)\twoheadrightarrow \Hom(\oDelta(\lambda),I/N)$.
But since $I$ is the sum of costandard objects, all homomorphisms from $\oDelta(\lambda)$ to $I$
factor through the socle of $I$ and hence through $N$. So $\Hom(\oDelta(\lambda),I/N)=0$ for all $\lambda$
and therefore $I/N=0$.  Thus, $N=I$ has a costandard filtration.

Part (3) follows immediately; consider an exact sequence
$0\rightarrow E_1\rightarrow E\rightarrow E_2\rightarrow 0$ in $\Cat_\xi$. We can assume by induction
that both $\nabla(E_1),\nabla(E_2)$ are $\onabla$-filtered.
The cokernel $N$ of $\nabla_\xi(E_1)\hookrightarrow \nabla_\xi(E)$ satisfies $\Ext^1(\Delta(\lambda),N)=0$
for all $\lambda$ and so is $\onabla$-filtered. So we have an embedding $N\hookrightarrow \nabla_\xi(E_2)$
of $\onabla$-filtered objects that becomes an isomorphism after projecting to $\Cat_\xi$. This forces
the costandard quotients of $N,\nabla_\xi(E_2)$ to be the same and the embedding to be an isomorphism
(recall that all blocks of $\Cat$ are finite).
%
\end{proof}

We also point out the following form of the BGG reciprocity.

\begin{Lem}\label{Lem:BGG}
The multiplicity $[P(\lambda):\Delta(\mu)]$ of $\Delta(\mu)$ in $P(\lambda)$ equals to
the multiplicity $[\onabla(\mu):L(\lambda)]$. Similarly, the multiplicity $[I(\lambda):\nabla(\mu)]$
equals $[\oDelta(\mu):L(\lambda)]$.
\end{Lem}

Combining these results, we see that:
\begin{Prop}
  The category $\Cat^{opp}$ is standardly stratified with respect to
  the map
  $\varrho:\Lambda\ \rightarrow \Xi$ and the standardization functor
  $\nabla_\xi:\Cat^{opp}_\xi\rightarrow \Cat^{opp}_{\leqslant \xi}$.
\end{Prop}

\section{Tensor product categorifications}

\subsection{Categorical \texorpdfstring{$\fg$}{g}-actions}
\label{sec:categ-texorpdfstr-ac}

Let $\g$ be a Kac-Moody algebra with its set $\{\alpha_i, i\in I\}$ of simple roots.
There are a variety of notions of {\bf categorical $\g$-actions} which
have appeared in the literature, in the work of Rouquier
\cite{Rou2KM}, Khovanov and Lauda \cite{KLIII}, Cautis and Lauda
\cite{CaLa} and others.  Of course, as with all definitions where
there is some flexibility, one endeavors to use the
weakest version possible when proving facts about objects
satisfying said definition and the strongest when showing that an
object does satisfy it (though one is often forced to do the
opposite).

All of these definitions employ the {\bf KLR algebra} or {\bf quiver Hecke
  algebra} $R$, a sum of finitely generated algebras $R_\mu$ attached to every
element $\mu$ in the positive cone of the root lattice of $\fg$; we
let $R_k$ for an integer $k$ denote the sum of the $R_\mu$'s for $\mu$
a sum of $k$ simple roots. We should note that our definition of KLR algebra follows that of
Rouquier \cite[\S 3.2]{Rou2KM}, and thus involves a choice of
polynomials $Q_{ij}(u,v)$ for each pair of elements in the Dynkin
diagram with degree in $u$ bounded above by the entry $-a_{ij}$ of the
negated Cartan matrix, and similarly the degree in $v$ bounded above
by $-a_{ji}$. We say this choice is {\bf homogeneous} if the polynomial
$Q_{ij}(u,v)$ is homogeneous when the ratio between the degrees of the
variables $u$ and $v$ equals the ratio between the lengths of the
simply roots $\alpha_i$ and $\alpha_j$.  We'll follow the conventions
of \cite[\S \ref{m-sec:categorification}]{Webmerged} throughout.

The finitely generated modules over the KLR algebra form a
monoidal category under induction functors; this monoidal category on its own
is a categorification, in a certain sense, of the enveloping algebra
$U(\fg_+)$ of the Borel $\fg_+$.

For our purposes, a {\bf categorical $\g$-action} on an
abelian $\C$-linear category $\Cat$ is an action
of the strict 2-category Rouquier denotes $\mathfrak{A}$; that is, it
consists of
\begin{itemize}
\item a module category structure over the representations of the KLR algebra
  generated by exact functors $F_i$; that is, a functor $F\cong \oplus F_i$
  such that $F^k$ carries an action of $R_k$.  In particular, each of the
  functors $F_i$ carries a
  natural transformation $y$, usually denoted as a dot in literature
  such as \cite{KLIII,Webmerged}, and
\item exact right adjoints $E_i$ to these functors, such that
\item   the map Rouquier denotes $\rho_{s,\mu}$ in \cite[\S
  4.1.3]{Rou2KM} is an isomorphism.
\end{itemize}
In particular, each pair of functors $E_i,F_i$ should be thought of a categorical
$\mathfrak{sl}_2$ action in the sense of Chuang and Rouquier.
All other notions of categorification mentioned above
are adding additional structure to this schema.

Since our main theorem will be a classification/uniqueness theorem, we
need to have a notion of equivalence between categorical actions.

\begin{defi}
  A {\bf strongly equivariant functor} between two categories
  $\mathcal{C}_1,\mathcal{C}_2$ with
  categorical $\fg$-actions is
  \begin{itemize}
  \item a functor $\eta\colon
    \mathcal{C}_1\to\mathcal{C}_2$ together with
  \item isomorphisms of functors $F\eta\cong \eta F$ which commute
    with the $R_k$-actions on $F^k\eta\cong \eta F^k$, such that the
    map $\eta E \to EF\eta E\cong
E\eta FE\to E\eta$ is also an isomorphism.
  \end{itemize}
  If we think of a categorical $\fg$-action as a representation of the
  2-category $\mathfrak{A}$ in the strict 2-category of $\C$-linear
  categories, this is the usual notion of natural transformation between
  representations of a 2-category.

  We call such a functor a {\bf strongly equivariant equivalence} if
  $\eta$ is an equivalence.
\end{defi}


Our starting point is the categorification of $ V$, an irreducible
representation of $\fg$ with highest weight $\nu$. As mentioned in the
introduction, there are several uniqueness theorems for such
categorifications, based on work of Rouquier \cite[\S 5.1]{Rou2KM}.
Since there are different contexts in which such representations
appear, we record here the version that we require. We use $R^\nu$ to
denote the cyclotomic KLR algebra for $\fg$ and the highest
weight $\nu$ (for example, as discussed
in \cite[\S \ref{m-sec:categorification}]{Webmerged}), and let $R^\nu_k$ denote the finite dimensional
subalgebra spanned by diagrams with $k$ strands, that is, the image of
$R_k$.

Assume $\EuScript{C}$ is an artinian and noetherian abelian $\C$-linear category.
\begin{Prop}\label{simple-unique}
Assume $\EuScript{C}$ has a $\fg$-action by exact functors such that:
\begin{enumerate}
\item the Grothendieck group $\mathbb{C}\otimes_{\Z} K^0(\EuScript{C})$ isomorphic to $ V$;
\item the subcategory $\EuScript{C}_{\nu}$ has a fixed equivalence to
  $\operatorname{Vect}_\C$, sending $\C$ to an object $\mathbb{V}$;
 \item the transformation $y$ acts nilpotently on $F_i$.
\end{enumerate}
Then $\EuScript{C}$ is strongly equivariantly equivalent to the
category of finite-dimensional
modules over $R^\nu$. The adjoint equivalences are given by
\[\big(\bigoplus_kF^k \mathbb{V}\big)\otimes_{R^\nu}\bullet\colon
R^\nu\operatorname{-mod}\to \EuScript{C},\qquad \bigoplus_k\Hom(F^k
\mathbb{V},\bullet)\colon \EuScript{C}\to  R^\nu\mmod.\]
\end{Prop}
\begin{proof}
Consider the category of projectives
$\EuScript{C}\operatorname{-proj}$; this is preserved by the functors
$E_i$ and $F_i$, since both have exact adjoints.  We can now apply
\cite[5.7]{Rou2KM} with
$\mathcal{V}=\EuScript{C}\operatorname{-proj}$.  This implies that
$\EuScript{C}\operatorname{-proj}$ is equivalent to a base change of
the universal categorification which Rouquier denotes
$\mathcal{L}(\nu)$.  By \cite[4.25]{RouQH}, this universal
categorification is the category of projective modules over a deformed
cyclotomic quotient $\mathcal{B}(\nu)$, and the base change precisely
consists of killing the deformation parameters $z_i$ to obtain the
category of 
projective modules over $R^\nu$.  The equivalence of additive
categories between the categories of projectives induces the desired
equivalence of abelian categories.
\end{proof}

We'll note, every irrep of the KLR algebra is absolutely irreducible,
since \cite[3.20]{KLI} implies that $\End(S_b)\cong \C$ for every
simple $S_b$.   This implies that in $\EuScript{C}$, every irreducible is
absolutely irreducible as well.

\subsection{Definition}\label{SS_definition}
Let $V_1,\ldots, V_n$ be irreducible integrable
$\g$-modules with highest weights $\bnu=(\nu_1,\dots, \nu_n)$ and $\C$
be an infinite field; we let $\nu=\nu_1+\cdots +\nu_n$. We are
going to define the notion of a categorification of the {\it ordered}
tensor product $V_1\otimes V_2\otimes\ldots\otimes V_n$.

As before, let $\Cat$ be an abelian artinian $\C$-linear category,
with each block equivalent to the representation category of a finite
dimensional $\C$-algebra. The data of a tensor product
categorification on $\Cat$ consists
of two parts:
\begin{itemize}
\item
a categorical $\g$-action on $\Cat$ in the sense of Rouquier where the
functors $E_i$ and $F_i$ are exact, 
  and the natural transformation $y$ acts locally nilpotently on $F_i$ and,
\item the structure of a standardly stratified category on $\Cat$ with
  poset $\Xi$.
\end{itemize}

These two pieces of data have to satisfy some compatibility conditions to be explained below.
\begin{itemize}
\item[(TPC1)] The poset $\Xi$ is the set of
  $n$-tuples $\bmu=(\mu_1,\ldots,\mu_n)$, where $\mu_i$ is a weight of
  $V_i$.  The poset structure is given by ``inverse dominance
  order'': we have
  \[\bmu=(\mu_1,\ldots,\mu_n)\geqslant \bmu'=(\mu'_1,\ldots,\mu'_n)\]
  if and only if $ \sum_{i=1}^n\mu_i=\sum_{i=1}^n \mu_i' $ and for all $1\leqslant j< n$, we have
\[\sum_{i=1}^j
  \mu_i\leqslant \sum_{i=1}^j \mu_i'.\]
  As usual, we write $\beta_1\leqslant
  \beta_2$ if $\beta_2-\beta_1$ is a linear combination of positive
  roots with non-negative coefficients.  We should note that the
  importance of order of tensor factors becomes immediately apparent in this definition.

\item[(TPC2)]
The associated graded category $\gr\Cat$ carries a categorical
$\fg^{\oplus n}$-action with $K^0(\gr\Cat)\cong V_1\otimes \cdots
\otimes V_n$ as $\fg^{\oplus n}$-modules such that the weight $\bmu$ subcategory of
$\gr\Cat$ is precisely the subquotient $\Cat_{\bmu}(=\Cat_{\leqslant  \bmu}/\Cat_{<\bmu})$.
We use $\,_{j}\!E_i$ and $\,_{j}\!F_i$
to denote categorification functors for the $j$th copy of $\fg$.  We
assume that $\Cat_\nu\cong \operatorname{Vect}_\C$;
we let
$\mathbb{V}$ denote the unique indecomposable object in this
subcategory. By Proposition \ref{simple-unique}, we thus have that
$\gr\Cat\cong \EuScript{C}$, the representations of the cyclotomic KLR
algebra of $\fg^{\oplus n}$ for the highest weight $\bnu$.

\item[(TPC3)] Finally, we must have a compatibility between the categorical
  $\fg$-action on $\Cat$ and the categorical $\fg^{\oplus n}$-action
  on $\gr\Cat$: for each $M\in \Cat_{\bmu}$, the object $E_i\Delta_{\bmu}(M)$ admits a filtration with successive quotients
being $\Delta(\,_{j}\!E_i M), j=1,\ldots,n$.
It is easy to see that such a filtration is determined uniquely,
we call it a {\it standard} filtration on $E_i \Delta_{\bmu}(M)$.

Similarly, we require that $F_i\Delta_{\bmu}(M)$ comes equipped with a filtration whose successive quotients
are $\Delta(\,_{j}\!F_i M), j=1,\ldots,n$.
\end{itemize}

Since every irreducible is absolutely irreducible in $\gr\Cat$ as we
noted above, this means that the same property will hold in any tensor
product categorification.

\begin{Rem}
Of course, we could try to make this definition more general by not
requiring $y$ to be nilpotent; however, this would not really gain us
any additional generality.  If $y$ acted on the functor $F_i$ with
more than one eigenvalue, it could then be split into
generalized eigenspaces for $y$, and these functors would give a pair
of
categorical actions of $\fg$.  Thus, we may as well assume that $y$
has only one eigenvalue $a$. Note
  that we can change this eigenvalue $a$ using the substitution $y\mapsto
  y-a$, at the cost of changing the relations of the KLR algebra.
  We must change the polynomials $Q_{ij}(u,v)$ by the same
  substitution.
\end{Rem}
\begin{Rem}
  Another point where the reader might wish to generalize this is to
  replace the condition that $\Cat_\nu \cong \operatorname{Vect}_\C$
  with the condition that (say) $\Cat_\nu$ is the representation
  category of a local Artinian $\C$-algebra $A$.  Our results should extend
  to this case, but at a considerable cost; in particular, the
  categorification obtained is no longer unique.  Rather, the
  possible choices of $\gr\Cat$ will have moduli, given by considering the minimal polynomial of $y$ for its
  induced action on ${}_jF_i\mathbb{V}$; the coefficients of this
  polynomial can be arbitrary elements of the radical of $A$, and one
  expects that there is a unique TPC with this choice of $\gr\Cat$.
  Aside from the intrinsic nuisance of working relative to $A$, there
  are two relatively minor, but non-trivial, technical obstacles here:
  \begin{itemize}
  \item there are competing definitions of categorical $\fg$-action,
    and it's not clear that they give the same result.
    The classification mentioned above in terms of minimal polynomials
    is known for the Cautis-Lauda 2-category from \cite{CaLa} by \cite[\ref{m-universal}]{Webmerged};
    Rouquier has announced the same result for his 2-category, but the
    proof has yet to appear.
 \item it is not actually proven that a TPC will exist in this
   relative case since the corresponding algebras are not considered
   in \cite{Webmerged}, though most results could be ported over by a
   careful use of Nakayama's Lemma.
  \end{itemize}
\end{Rem}
As in any standardly stratified category, we have an isomorphism of
Grothendieck groups $K^0(\gr \Cat)\cong K^0(\Cat)$ via
the map sending $[M]\mapsto [\Delta(M)]$.  By assumption, we obtain an isomorphism $K^0(\Cat)\cong V_1\otimes \cdots
\otimes V_n$. It follows immediately from (TPC3) that:
  \begin{Prop}
    For any tensor product categorification, this map is an
    isomorphism of $\g$-modules.
  \end{Prop}
We also note that any tensor product categorification is
integrable, so by \cite[5.16]{Rou2KM}, the functors $E_i$ and $F_i$
are necessarily biadjoint.
\begin{Rem}
In fact, we could give an axiomatic description of a tensor product of arbitrary $\g$-categorifications
$\Cat^1,\ldots,\Cat^n$. Let us elaborate on this in the case when $n=2$. We say that a $\g$-categorification
$\Cat$ equipped with a standardly stratified structure is the tensor product $\Cat^1\otimes\Cat^2$ if
\begin{itemize}
\item[(TPC1')]
the poset $\Xi$ of $\Cat$ is the set of pairs $(\nu_1,\nu_2)$, where $\nu_i$ is a weight for $\Cat^i$,
with the order given by $(\nu_1,\nu_2)\leqslant (\nu_1',\nu_2')$ if $\nu_1\geqslant \nu_1'$ and
$\nu_1+\nu_2=\nu_1'+\nu_2'$.
\item[(TPC2')] There is an identification of $\gr\Cat$ with $\Cat^1\boxtimes \Cat^2$.
\item[(TPC3')] For each $M\in \gr\Cat$, the object $E_i\Delta(M)$
  admits a short exact sequence \[0 \longrightarrow \Delta(E_i^2 M) \longrightarrow
  E_i\Delta(M) \longrightarrow \Delta(E_i^1 M) \longrightarrow 0\] and
  similarly for the functors $F_i$, since $(\nu_1+\al_i,\nu_2)<(\nu_1,\nu_2+\al_i)$ in reverse dominance order.
\end{itemize}
Unfortunately, in general, we can prove neither existence nor
uniqueness of such tensor products; we expect that they will arise from
Rouquier's proposed internal tensor product.
\end{Rem}

\subsection{Labeling}\label{sec:labeling}
 Before getting too deep into the structure of these categories, we
 should give a set of labels for simples (or indecomposable projectives) in $\Cat$.

As usual with a standardly stratified category, the simples (or indecomposable projectives) in $\Cat$ are in canonical
bijection with the simples (or indecomposable projectives) in $\gr
\Cat$.  Let $B(\nu_j)$ be the crystal of the irreducible
representation $V_j$.  By \cite[\S 5.1]{LV}, we have a canonical
bijection between the product $B(\nu_1)\times \dots \times B(\nu_n)$
and the set of simples (or indecomposable projectives) in $\gr\Cat$.
Recall that the set of simple
objects in an arbitrary $\g$-categorification has a $\g$-crystal structure: if $L$ is a simple, then for $\tilde{e}_i L$ we take the
head (equivalently, the socle) of the object $E_i L$ if the latter is nonzero and $0$ else. The crystal operator
$\tilde{f}_i$ is defined similarly using $F_i$; as usual, we use the
notation ${}_j \tilde{e}_i,{}_j \tilde{f}_i$ when considering these
operators for $\fg^{\oplus n}$.


First consider the case of the categorification of a simple module
with highest weight $\nu$, which, as before, we denote
$\EuScript{C}$.  One straightforward description of the projective for a
crystal element uses the string parameterization of vertices of $B(\nu)$. Consider
$\lambda\in B(\nu)$ for some $j=1,\ldots,n$, and let $P(\la)$ be the
associated projective.  Choose an infinite
sequence $i_1,i_2,\dots$ of nodes in the Dynkin diagram of $\fg$
containing each node infinitely many times.  The {\bf string
  parameterization} of $\lambda$ is the unique  infinite sequence of integers
$(a_1,a_2,\dots)$ with almost all
entries 0 such that \[\te_{i_j}^{a_j}\cdots \te_1^{a_1}\la\neq 0
\qquad \te_{i_j}^{a_j+1}\te_{i_{j-1}}^{a_{j-1}}\cdots \te_1^{a_1}\la=
0\qquad \text{for all $j$.}\]

We can order crystal elements by comparing string parametrizations
lexicographically.
\begin{Prop}
[\mbox{Khovanov-Lauda \cite[3.20]{KLI}}]
  The projective $P(\la)$ is the unique summand of
  $F_{i_1}^{a_1}F^{a_2}_{i_2}\cdots \mathbb{V}$ which doesn't appear
  in $F_{i_1}^{a_1'}F^{a_2'}_{i_2}\cdots \mathbb{V}$ for a word
  $\mathbf{a}'$ larger than $\mathbf{a}$ in
  lexicographic order.
\end{Prop}
For $\fg^{\oplus n}$, we want to proceed a little differently; instead
of applying this construction directly (which will work perfectly
well), we compute the string parameterization of each component of
$\bla=(\la_1,\dots, \la_n)\in B(\nu_1)\times \dots \times B(\nu_n)$.
Thus, we obtain $n$ different words $\Ba^{(1)}=(a_1^{(1)},\dots)$,
etc.  We can easily modify the proposition above to:
\begin{Prop}
  The projective $P(\bla)$ is the unique summand of
  ${}_{1}F_{i_1}^{a_1^{(1)}}\!\!{}_{1}F^{a_2^{(1)}}_{i_2}\!\!\cdots {}_{2}F_{i_1}^{a_1^{(2)}}\!\!{}_{2}F^{a_2^{(2)}}_{i_2}\!\!\cdots\mathbb{V}$ which doesn't appear
  in a corresponding monomial where
  any $\mathbf{a}^{(j)}$ is replaced by a larger word in
  lexicographic order.
\end{Prop}

\subsection{Tensor product categorification on $\Cat^{opp}$}
Suppose that $\Cat$ is a tensor product categorification of $V_1\otimes\ldots\otimes V_n$ in the sense
of the definition above.  In this subsection we are going to prove
that $\Cat^{opp}$ is also a tensor
product categorification of $V_1\otimes V_2\otimes\ldots\otimes V_n$. 

\begin{Prop}\label{Lem:opp_catn}
The category $\Cat^{opp}$ is a tensor
product categorification of $V_1\otimes\ldots\otimes V_n$
\end{Prop}
\begin{proof}
Conditions (TPC1)  is tautologically equivalent for
$\Cat$ and $\Cat^{opp}$. To prove that (TPC2) holds for $\Cat^{opp}$, we notice that Proposition \ref{simple-unique}
provides a strongly equivariant equivalence of $\gr (\Cat^{opp})=(\gr\Cat)^{opp}$
and $\gr\Cat$.  Thus, we need only establish (TPC3).

We need to prove that $E_i\nabla_\bmu(M)$ has a filtration whose successive quotients are
$\nabla_{\bmu+\alpha^j_i}(\,_{j}\!E_i M), $ and the analogous claim for $F_i \nabla_\bmu(M)$. Here $\alpha_i^j$
means the simple root $\alpha_i$ for the $j$th copy of $\g$.

First of all,  $E_i,F_i$ preserve $\Cat^{\onabla},\Cat^\nabla$. This follows from the  observation
that $E_i,F_i$ preserve $\Cat^{\Delta},\Cat^{\oDelta}$, combined with Lemma \ref{Lem:cost_filt} and the biadjointness
of $E_i$ and $F_i$.

Pick $N\in \Cat_{\bmu'}$, where $\bmu'-\bmu=\alpha_i^j$. We see that
\begin{align}\label{eq:Hom_eq}&\Hom_{\Cat}(\Delta_{\bmu'}(N),E_i\nabla_\bmu(M))= \Hom_{\Cat}(F_i\Delta_{\bmu'}(N),\nabla_\bmu(M))=\\\nonumber &=\Hom_{\Cat_\bmu}(\,_{j}\!F_i N, M)=\bigoplus_{j=1}^n \Hom_{\Cat_{\bmu'}}(N, \,_{j}\!E_iM)=\\\nonumber&=\Hom_{\Cat}(\Delta_{\bmu'}(N), \nabla_{\bmu'}(\,_{j}\!E_i M)), \end{align}
where all equalities are natural isomorphisms of $\operatorname{End}_{\Cat_{\bmu'}}(N)$-modules.

Now let us show that the claim in the beginning of the proof holds when $M$ is projective (=injective)
in $\Cat_\bmu$. Recall that $E_i \nabla_\bmu(M)\in \Cat^\nabla$.
In particular,  if, in (\ref{eq:Hom_eq}), for $N$ we take the simple in $\Cat_{\bmu'}$ labeled by $\lambda$,
we  see that the multiplicity of $\nabla(\lambda)$ in $E\nabla_\bmu(M)$ and $\bigoplus_{i=1}^n \nabla_{\bmu+\alpha_i^j,\bmu'}(\,_{j}\!E_i M)$
coincide. This implies the existence of a required filtration on $E_i\nabla_\bmu(M)$.

Proceed to the case of a general $M$. In this case, Lemma \ref{Lem:cost_filt} just implies that
$E_j\nabla(M)\in \Cat^{\oDelta}$. So, by (\ref{eq:Ext_vanish}),
the object $E_j\nabla(M)$ has a filtration
with successive quotients $\nabla(M')$,  $M'\in\Cat_{\bmu'}$,  for $\bmu'=\bmu+\alpha_i^j, j=1,\ldots,n$.
Recall that we write $\iota_{\bmu'}$ for the inclusion functor
$\Cat_{\leqslant \bmu'}\hookrightarrow \Cat$, $\pi_{\bmu'}$ for the projection
functor $\Cat_{\leqslant \bmu'}\twoheadrightarrow \Cat_{\bmu'}$, and $\iota_{\bmu'}^!,\pi_{\bmu'}^!$
for the left adjoint functors. Choose $j=1,\ldots,n$ and set $\bmu'=\bmu+\alpha_i^j$. Consider the functor $\mathcal{F}_{\bmu'}:=\pi_{\bmu'}^!\circ\pi_{\bmu'}\circ\iota_{\bmu'}^!:\Cat\rightarrow
\Cat_{\leqslant \bmu'}$. It follows from Lemma \ref{Lem:proj_exact} that the functor $\mathcal{F}_{\bmu'}$ is exact
on $\Cat^{\oDelta}$. It maps $M'\in \Cat^{\oDelta}$ to the subquotient of the form $\Delta_{\bmu'}(N')$, where
$N'=\pi_{\bmu'}\circ\iota_{\bmu'}^!(M')$. So  we just need to prove that the functor $\pi_{\bmu'}\circ\iota_{\bmu'}^!\circ E_i\circ\Delta_\bmu:
\Cat_{\bmu}\rightarrow \Cat_{\bmu'}$ is isomorphic  to $\,_{j}\!E_i$ if $\bmu'=\bmu+\alpha_i^j$ for some $j=1,\ldots,n$
and is zero else. The vanishing result follows from the form of $E_i\nabla_\bmu(M)$, for $M$ projective,
obtained above.  The isomorphism of functors follows from (\ref{eq:Hom_eq}).
\end{proof}

\subsection{Relation with previous constructions}
\label{sec:relat-with-prev}

  Concrete examples of categorical $\g$-actions whose Grothendieck
groups are tensor products have arisen in work in
representation theory and topology.

\subsubsection*{Diagrammatic realizations} One obvious construction to
compare the definition above with are the algebras $T^\bnu$ defined by the
second author in \cite[\S \ref{m-sec:KL}]{Webmerged}.  We refer the reader to that paper for
the details of the definition.
What is important for us is an inductive description of the
representation categories of these algebras.  Given the sequence of weights $\bnu$, we define weights
$\nu^{(k)}=\nu_1+\cdots +\nu_k$.

Attached to each $k$, we have an associated cyclotomic
quotient of the KLR algebra \[R^{\nu^{(k)}}=R/\langle y_1^{\alpha^\vee_{i_1}(\nu^{(k)})}e(\Bi)|\Bi\in I^m\rangle,\] equipped with projections
$R^{\nu^{(k)}}\to R^{\nu^{(k-1)}}$, and induced inflation functors
$\operatorname{inf}_k\colon  R^{\nu^{(k-1)}}\mmod\to  R^{\nu^{(k)}}\mmod$.

Now, consider the category $\Cat(\bnu)$ defined as the category of
representations of its category $\Cat(\bnu)\operatorname{-proj}$ of projectives (via the Yoneda embedding):
\begin{itemize}
\item we let $\Cat(\nu)$ just be the category of finite dimensional
  representations of $R^{\nu}$.
\item The category of projectives $\Cat(\bnu)\operatorname{-proj}$ is
  the additive category of generated by summands of categorification
  functors applied to the image
  $\operatorname{inf}_k(\Cat(\nu_1,\dots, \nu_{k-1})\operatorname{-proj})$.
\end{itemize}
Thus, these are the minimal subcategories closed under the categorical
$\fg$-action which contains the images of all inflation functors.

\begin{Prop}
  We have an equivalence $T^{\bnu}\mmod\cong \Cat(\bnu)$.
\end{Prop}
\begin{proof}
  This follows from \cite[\ref{m-prop:add-embed}]{Webmerged}.  In that paper we define a
  fully faithful functor $T^\bnu\operatorname{-pmod}\to R^\nu\mmod$
  whose essential image is exactly the
  additive category generated by summands of $y_{\Bi,\nu}R^\nu$
for
  certain elements $y_{\Bi,\kappa}$ associated to sequences $\Bi=(i_1,\dots, i_m)$ and functions $\kappa\colon
  [1,n]\to [0,m]$. 
If $\kappa(n)\neq m$, then
  $y_{\Bi,\kappa}T^\nu=F_{i_m}(y_{\Bi^-,\kappa}T^\nu)$ where
  $\Bi^-=(i_1,\dots, i_{m-1})$.  Thus, the image of $T^{\bnu}\mmod$ is
  generated by categorification functors applied to the modules  $y_{\Bi,\kappa}T^\nu$
  with
  $\kappa(n)= m$.  Since these modules are  exactly obtained by
  applying the inflation functor $\operatorname{inf}_{n-1}$ to
  the images of
  $T^{(\nu_1,\dots, \nu_{n-1})}\operatorname{-pmod}$, we
  are done.
\end{proof}

The most important fact for us is that:
\begin{Thm}[\mbox{\cite[\ref{m-TPC}]{Webmerged}}]
  The category $\Cat(\bnu)$ with its standardly stratified structure from
  \cite[\ref{m-SS}]{Webmerged} and categorical $\fg$-action from
  \cite[\ref{m-full-action}]{Webmerged} is a tensor product categorification for
  $V_1\otimes \cdots \otimes V_n$.
\end{Thm}
\begin{proof}
  We consider the axioms of a tensor product categorification in turn,
  and confirm them:
\begin{itemize}
\item[(TPC1)] We must have that  the poset underlying the
  stratification is that of
  $n$-tuples $\bmu=(\mu_1,\ldots,\mu_n)$, where $\mu_i$ is a weight of
  $V_i$.  The poset structure is given by ``inverse dominance
  order'': we have
  \[\bmu=(\mu_1,\ldots,\mu_n)\geqslant \bmu'=(\mu'_1,\ldots,\mu'_n)\]
  if and only if $ \sum_{i=1}^n\mu_i=\sum_{i=1}^n \mu_i' $ and for all $1\leqslant j< n$, we have
\[\sum_{i=1}^j
  \mu_i\leqslant \sum_{i=1}^j \mu_i'.\]
This precisely matches the definition of the order on root functions
from \cite[\S \ref{m-sec:standard-def}]{Webmerged}, since $\mu_i=\la_i-{\bf \alpha}(i)$.
The standardization
functors are exact by \cite[\ref{m-standard-exact}]{Webmerged} and $\Cat(\bnu)$
is standardly stratified in the sense of \cite{CPS96} by \cite[\ref{m-SS}]{Webmerged}.

\item[(TPC2)]
From \cite[\ref{m-semi-orthogonal}]{Webmerged} we see that the subquotients of this
standardly stratified structure are equivalent to
$\mathcal{C}(\nu_1)\otimes \cdots
 \otimes \mathcal{C}(\nu_n)$ and thus carry the expected categorical $\fg^{\oplus n}$ action
on these subquotients.

\item[(TPC3)]  The filtration of  \cite[\ref{m-prop:act-filter}]{Webmerged} shows that $E_i$
and $F_i$ acting on $\Delta(M)$ have the desired filtrations.\qedhere
\end{itemize}
\end{proof}

Tensor product categorifications for $\sl_m$ also arise in  more classical representation theory.
Here, we  give two examples.

\subsubsection*{Category $\mathcal{O}$} Consider the Lie algebra $\gl_N(\mathbb{C})$, its parabolic subalgebra with blocks (from top to bottom) of sizes
$m_1,\ldots,m_n$ and also fix a positive integer $n$. Let $\OCat$ be the corresponding parabolic category
$\mathcal{O}$. The integral blocks of this category form a highest weight category whose standard
objects are parabolic Verma modules \[\Delta(\lambda)\text{ with
}\lambda=(\lambda_1,\cdots,\lambda_N) \text{ for }
\lambda_1>\lambda_2>\cdots>\lambda_{m_1},\lambda_{m_1+1}>\cdots>\lambda_{m_2},\cdots\]
of highest weight $\rho + \sum\lambda_i\epsilon_i$.

\begin{defi}
  Let $\Cat(\mathbf{m})$ be the the sum of blocks of $\OCat$ spanned by
  $\Delta(\lambda)$ with $\lambda_i\in \{1,\ldots,m\}$.
\end{defi}

This category  is a tensor product categorification of
$\bigwedge^{m_1}\C^m\otimes\bigwedge^{m_2}\C^m\otimes \cdots\otimes \bigwedge^{m_n}\C^m$:
checking (TPC1) and (TPC2) is easy, while (TPC3) follows, for example, from \cite[4.3]{LoHWCII}. It was shown
in \cite[\ref{m-equiv}]{Webmerged} that the category $\Cat$ is strongly equivariantly equivalent to $\Cat(\omega_{m_1},\ldots,\omega_{m_n})$,
where $\omega_i$ is the $i$th fundamental weight. The main theorem of
this paper also provides a new proof of this equivalence.

\subsubsection*{Representations of $GL_n$} Often we can also realize tensor product categorifications as subquotients of interesting categories.
Let us give an example when the field $\C$ has characteristic $p>0$ and the algebra $\g$ acting
is $\sl_{p}$. Consider the category $\tilde{\Cat}=\bigoplus_{n=0}^{+\infty} \tilde{\Cat}_n$, where $\tilde{\Cat}_n$
is the category of polynomial representations of $\operatorname{GL}_k$ of degree $n\leqslant k$.
This is a highest weight category, whose labeling poset is that of
partitions (with respect to the $p$-dominance ordering).

A categorical $\hat{\sl}_p$-action on this category was first
introduced in \cite{HY}, and  this action is highest
weight in the sense of \cite{LoHWCII}. Fix a residue $r$ and consider the subalgebra $\sl_{p}
\subset \hat{\sl}_p$ corresponding to the other $p-1$ residues.

We introduce an equivalence relation
$\sim_r$ on the set of Young diagrams: $\lambda\sim_r\mu$ if the boxes in $\lambda$ and $\mu$ with residue
$r$ are the same.   Attached to each such equivalence class is a list
of coordinates $(x_0,y_0), (x_1,y_1),\cdots, (x_\ell,x_{\ell})$ given by the rightmost box
in each diagonal of the partition diagram with content congruent to
$r\pmod p$ listed left to right; we must also include the first empty
diagonals encountered on the left and right, that is, we have $x_0=0$
and $y_\ell=0$.  We let $m_i=y_i -y_{i+1}$, and note that $0\leqslant
m_i\leqslant p$.

Each equivalence class is an interval in the highest weight poset of $\tilde{\Cat}$,
so for an equivalence class $e$, we can consider the subquotient category $\Cat_e$ corresponding to $e$.
This is  a highest weight category with a well-defined highest weight
categorical action of $\sl_{p}$ and a tensor product categorification of the product
$\bigwedge^{m_1}\C^{p}\otimes\bigwedge^{m_2}\C^{p}\otimes \cdots\otimes \bigwedge^{m_n}\C^{p}$.

For example, if $p=3,r=0$ and $\lambda=(7,5,1^5)$
\begin{figure}
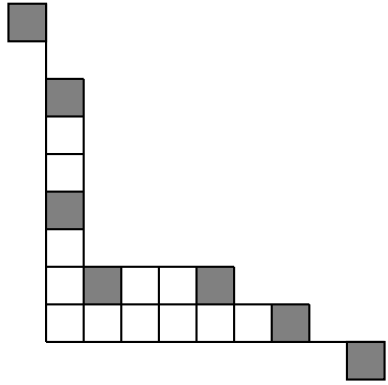

  \centering
  \tikz[thick,scale=.5]{
\fill[gray] (1,1) -- (1,2) --(2,2) -- (2,1) --cycle;
\fill[gray] (1,4) -- (1,3) --(0,3) -- (0,4) --cycle;
\fill[gray] (4,1) -- (5,1) --(5,2) -- (4,2) --cycle;
\fill[gray] (1,7) -- (1,6) --(0,6) -- (0,7) --cycle;
\fill[gray] (7,1) -- (6,1) --(6,0) -- (7,0) --cycle;
\fill[gray,draw=black] (-1,9) -- (-1,8) --(0,8) -- (0,9) --cycle;
\fill[gray, draw=black] (9,-1) -- (8,-1) --(8,0) -- (9,0) --cycle;
\draw (0,0) -- (9,0);
\draw (0,0) -- (0,9);
\draw (0,1) --(7,1);
\draw (1,0) --(1,7);
\draw (0,2) --(5,2);
\draw (2,0) --(2,2);
\draw (0,3) --(1,3);
\draw (3,0) --(3,2);
\draw (0,4) --(1,4);
\draw (4,0) --(4,2);
\draw (0,5) --(1,5);
\draw (5,0) --(5,2);
\draw (0,6) --(1,6);
\draw (6,0) --(6,1);
\draw (0,7) --(1,7);
\draw (7,0) --(7,1);
}
  \caption{The example with $p=3,r=0$ and $\lambda=(7,5,1^5)$.  The
    boxes $\{(9,0),(7,1),(5,2),(2,2),(1,4),(1,7),(0,9)\}$ are marked.}
\end{figure}
, we have the sequence of boxes \[\{(9,0),(7,1),(5,2),(2,2),(1,4),(1,7),(0,9)\}\]
and the tensor product categorified is  $\bigwedge^{2}\C^3\otimes \bigwedge^{2}\C^3\otimes \C^3\otimes \C^3$.

\subsubsection*{Quotient categories} When $\g$ is
$\widehat{\sl}_m$, then by \cite[\ref{m-q-Schur}]{Webmerged} the category
$\Cat(\omega_{m_1},\ldots,\omega_{m_n})$ is equivalent to an appropriate quotient
of the category of representations of cyclotomic $q$-Schur algebras,
where $q$ is a primitive $m$th root of $1$, with parameters
corresponding to the residues $m_1,\ldots,m_n$. However, it is not so easy to see the structure of a tensor product
categorification on the latter from the beginning, this should require an appropriate modification of the splitting
procedure explained below.

Also we would like to point out that tensor product categorifications for non-fundamental weights can be realized as
quotients of parabolic categories $\mathcal{O}$ (for finite type $A$) or of the representation categories
of cyclotomic Schur algebras. Again, it is easier to see this if one works with categories $\Cat(\nu_1,\ldots,\nu_n)$.

\begin{Prop}\label{Prop:condens}
  For any tensor product categorification $\Cat(\nu_1,\ldots,\nu_n)$,
  and list of indices $1\leqslant j_1< \cdots <j_m< n$, we have an exact
  quotient functor \[p_{\mathbf{j}}\colon \Cat(\nu_1,\ldots,\nu_n)\to \Cat(\nu_1+\cdots
  +\nu_{j_1},\nu_{j_1+1}+\cdots
  +\nu_{j_2},\dots,\nu_{j_m+1}+\cdots+\nu_n)\] categorifying the
  natural projection of tensor product representations.
\end{Prop}
\begin{proof}
  By \cite[\ref{m-split-strands}]{Webmerged}, there's an idempotent $e_{\mathbf{j}}$ in
  $T^\bnu$ such that $e_{\mathbf{j}}T^{\bnu}e_{\mathbf{j}}\cong
  T^{\bnu_{\mathbf{j}}}$ where $\bnu_{\mathbf{j}}=(\nu_1+\cdots
  +\nu_{j_1},\dots,\nu_{j_m+1}+\cdots+\nu_n)$.  This is
  the sum of the diagrams with no crossings, no dots, and no black
  strands between the red strands corresponding to weights we have
  condensed.

  Thus, the desired exact functor $T^{\bnu}\mmod\to
  T^{\bnu_{\mathbf{j}}}\mmod$ is just multiplication by $e_{\mathbf{j}}$.
\end{proof}

\section{Categorical splitting}
\subsection{Setting}

As before, fix simple $\g$-modules $V_i$ with highest weight $\nu_i$.
Consider a tensor product categorification $\Cat$
of $V_1\otimes V_2\otimes\dots\otimes V_n$ with poset
$\Xi=\{\bmu=(\mu_1,\ldots,\mu_n)\}$. Set $\Xi_0:=\{\mu\in \Xi|
\mu_n<\nu_n\}$.
We intend to study the subcategory $\Cat_{\Xi_0}$ and the
corresponding quotient $\underline{\Cat}^+:=\Cat/\Cat_{\Xi_0}$. We remark that, under our identification $[\Cat]\cong V_1\otimes V_2\otimes\dots\otimes V_n$, the space
$[\underline{\Cat}^+]$ is naturally identified with $V_1\otimes V_2\otimes \dots\otimes V_{n-1}\otimes v_{\nu_n}$,
where $v_{\nu_n}$ is the singular vector in $V_n$. We write $\Cat^i$ for the categorification of the simple $V_i$.
Further, for $\lambda_i\in B(\nu_i)$, by $v_{\lambda_i}$ we denote the class of the corresponding simple
object in $[\Cat^i]=V_i$.

Our goal in this section is to produce a categorical $\g$-action on the quotient $\underline{\Cat}^+:=\Cat/\Cat_{\Xi_0}$ making it into
a tensor product categorification of $V_1\otimes V_2\otimes\dots\otimes V_{n-1}$. Since $\Xi_0$ is a poset ideal,
the quotient category has a natural standardly stratified structure that we are going to use.
The corresponding poset will be denoted by $\underline{\Xi}^+$, the poset associated to $V_1\otimes V_2\otimes\dots\otimes V_{n-1}$
in the same way as $\Xi$ is associated to $V_1\otimes V_2\otimes\dots\otimes V_n$.

The following lemma shows that half of categorification functors act
in a straight-forward way on $\underline{\Cat}^+$.

\begin{Lem}\label{Lem:F_stab}
The subcategory $\Cat_{\Xi_0}$ is stable with respect to the functors $F_i$.
\end{Lem}
\begin{proof}
Any simple in $\Cat_{\Xi_0}$ is a composition factor of $\Delta(\lambda)$ with $\varrho(\lambda)\in \Xi_0$.
So it is enough to show that $F_i\Delta(\lambda)\in \Cat_{\Xi_0}$. This is an easy corollary of (TPC3).
\end{proof}

So we have the induced functor $\underline{F}_i^+$ on $\underline{\Cat}^+$.
The same condition (TPC3) shows that, in general, $\Cat_{\Xi_0}$ is not closed with respect to
$E_i$. In order to get the functor $\underline{E}_i^+$ on $\underline{\Cat}^+$ we will need
to truncate the functor $E_i$. Our construction will generalizes
\cite[Section 5]{LoHWCII}, where the first author studied essentially the case of the tensor
products of the tautological $\sl_2$-modules.

\subsection{An equivalence}
We will construct the functors $\underline{E}_i^+$ one simple root at
a time.  Thus, throughout the remainder of this section, we fix an
index $i\in I$.

First, we are going to take a minimal, in a way, standardly stratified quotient of $\Cat$, where $E_i$ is well-defined.
Namely, consider $\Xi_1=\{\bmu\in \Xi| \nu_n-\mu_n\not\in \Z\alpha_i\}$, clearly, $\Xi_1$ is a poset
ideal contained in $\Xi_0$. An analogous argument to Lemma
\ref{Lem:F_stab} shows that:
\begin{Lem}
 The subcategory  $\Cat_{\Xi_1}$ is stable with respect to both $E_i,F_i$ and so we
  have well-defined functors on
  $\underline{\Cat}(=\underline{\Cat}_i):=\Cat/\Cat_{\Xi_1}$ again
  denoted by $E_i,F_i$.
\end{Lem}
Set $\underline{\Xi}=\Xi/\Xi_1$, this is the poset of the standardly stratified
category $\underline{\Cat}$.
We will need a standardly stratified subcategory $\underline{\Cat}^-(=\underline{\Cat}^-_i)$ inside
$\underline{\Cat}$. Namely, consider the poset ideal $\underline{\Xi}^-\subset \underline{\Xi}$
consisting of all $\bmu$ with $\mu_n=s_i\nu_n=\nu_n-r\al_i$, where
$s_i$ is the simple reflection corresponding to $i$ and $r=\al_i^\vee(\nu_n)$.
Let $\underline{\Cat}^-\subset \underline{\Cat}$ be the subcategory corresponding to $\underline{\Xi}^-$.
Under our identification $[\Cat]\cong V_1\otimes\ldots\otimes V_n$, the  complexified Grothendieck group
$[\underline{\Cat}^-]$ is identified with $V_1\otimes V_2\otimes\ldots\otimes V_{n-1}\otimes v_{\mu_n}$. Let $\iota$ be the inclusion $\underline{\Cat}^-\hookrightarrow \underline{\Cat}$,
and  $\pi:\underline{\Cat}\twoheadrightarrow \underline{\Cat}^+$ be the projection.
\begin{Prop}
  The functor $\Efun=\pi\circ
  E_i^{(r)}\circ\iota:\underline{\Cat}^-\rightarrow
  \underline{\Cat}^+$ is an equivalence of standardly stratified
  categories with quasi-inverse given by $\Ffun:=\iota^!\circ F_i^{(r)}\circ\pi^!$.
  Here $r=\al_i^\vee(\nu_n)$.
\end{Prop}
The proof closely follows that in \cite[Section 5.2]{LoHWCII} but we are going to provide it for readers convenience.
So far, we notice that, by the construction, $\Efun$ is exact and $\Ffun$ is left adjoint to $\Efun$.

We start by establishing some basic properties of $\Efun$, compare with \cite[Lemma 5.1]{LoHWCII}.
Thanks to results of Chuang and Rouquier, \cite{CR04}, for $\bmu\in \underline{\Xi}^-$ the functors
${}_nE_i^{(r)},{}_nF_i^{(r)}$ restrict to quasi-inverse equivalences between $\Cat^n_{s_i\nu_n},
\Cat^n_{\nu_n}$. We will identify these categories using the functors. Also we identify the posets
$\underline{\Xi}^\pm$ with the poset associated to $V_1\otimes\ldots\otimes V_{n-1}$.  Finally, we identify
$[\underline{\Cat}^\pm]$ with $V_1\otimes V_2\otimes\ldots \otimes V_{n-1}$ by sending $\oDelta(\lambda_1,\ldots,\lambda_{n-1},\lambda_n)$
(where, recall, $\lambda_n$ is the label of the only simple object in $\Cat^n_{\nu_n}$ or $\Cat^n_{s_i\nu_n}$)
to $v_{\lambda_1}\otimes\ldots\otimes v_{\lambda_{n-1}}$. Below for $\bla=(\lambda_1,\ldots,\lambda_{n-1})$
we write $L_\pm(\bla), \Delta_\pm(\bla)$ etc. for the corresponding objects in $\underline{\Cat}^\pm$.


\begin{Lem}\label{Lem:Efun_prop}
\mbox{}
\begin{enumerate}
\item The functor $\Efun$ intertwines the standardization functors $\Delta_\bmu$, where $\bmu=(\mu_1,\ldots,\mu_{n-1})$.
\item The functor $\Efun$ intertwines the costandardization functors $\nabla_\bmu$.
\item The induced map $[\Efun]:[\underline{\Cat}^-]\rightarrow [\underline{\Cat}^+]$ is the identity.
\item $\Efun(L_-(\lambda))=L_+(\lambda)$.
\end{enumerate}
\end{Lem}
\begin{proof}
Let us prove (1).
Set $\bmu_-:=(\mu_1,\ldots,\mu_{n-1},s_i\nu_n), \bmu_+:=(\mu_1,\ldots,\mu_{n-1},\nu_n)$.
The object $E_i^r\Delta_{\bmu_-}(N)$ is $\oDelta$-filtered. Moreover, applying condition (TPC3) in the definition
of a tensor product categorification $r$ times we get a filtration on $E_i^r \Delta_{\bmu_-}(N)$ whose
successive quotients looks as follows: $\Delta_{\bmu_-+\alpha}(E_i^{\alpha}N)$, where $\alpha=\sum_{j=1}^n m_j \alpha_i^j$
with $\sum m_j=r$, and $E_i^{\alpha}:=\,_1\!E_i^{m_1}\boxtimes \,_2\!E_i^{m_2}\boxtimes\ldots\boxtimes \,_n\!E_i^{m_n}$, and that quotient appears $r!\binom{r}{m_1,\ldots,m_n}$ times. The only quotient that survives under $\pi$ is the subobject $\Delta_{\bmu_+}(\,_n\!E_i^r N)$.
The functor $\pi \circ E_i^r\circ\iota$ is isomorphic to $\Efun^{\oplus r!}$, and $\,_n\!E_i^r$ is isomorphic to the sum of $r!$
copies of our identification $\Cat_{\bmu_-}\hookrightarrow \Cat_{\bmu_+}$. We deduce that the functors $\Efun\circ \Delta_{\bmu,-}^{\oplus r!}$
and $\Delta_{\bmu,+}^{\oplus r!}$ are isomorphic. We claim that this implies that $\Efun\circ \Delta_{\bmu,-}\cong \Delta_{\bmu,+}$.
It is enough to prove an isomorphism on each block separately. Since all weight spaces in $V_1\otimes\ldots\otimes V_{n-1}$
are finite dimensional, the blocks of $\underline{\Cat}^{\pm}$ are isomorphic to categories of finite dimensional modules
of finite dimensional algebras. The functors $\Efun\circ\Delta_{\bmu,-},\Delta_{\bmu,+}$ are right exact and so are given
by tensor products with bimodules, say $B_1,B_2$. We know that $B_1^{\oplus r!}\cong B_2^{\oplus r!}$ and hence,
by the Krull-Schmidt theorem, $B_1\cong B_2$.

%

(2) follows from (1) applied to $\Cat^{opp}$ or can be proved completely analogously to (1).
(3) is a direct corollary of (1) and the particular form of the identification  $[\underline{\Cat}^+]
\cong [\underline{\Cat}^-]$. To prove (4) we notice that $L_-(\lambda)$ is the image of any nonzero
morphism $\varphi:\oDelta_-(\lambda)\rightarrow \onabla_-(\lambda)$. Since $\Efun$ is exact, we see that
$\Efun(L_-(\lambda))$ is the image of $\Efun(\varphi):\Efun(\oDelta_-(\lambda))\rightarrow \Efun(\onabla_-(\lambda))$.
Thanks to (1) and (2), $\Efun(\oDelta_-(\lambda))=\oDelta_+(\lambda),\Efun(\onabla_-(\lambda))=\onabla_+(\lambda)$.
So $\Efun(L_-(\lambda))$ is either $L_+(\lambda)$ or $0$. The latter is impossible because of (3).
\end{proof}

Now let us list some basic properties of $\Ffun$, compare with \cite[Lemmas 5.2,5.3]{LoHWCII}.

\begin{Lem}\label{Lem:Ffun_prop}
\mbox{}
\begin{enumerate}
\item We have an isomorphism $\Ffun(P_+(\bla))\cong P_-(\bla)$ for any $\bla=(\lambda_1,\ldots,\lambda_{n-1})$.
\item The natural morphism $\Ffun\circ \Efun(M)\rightarrow M$ is surjective for any $M\in \underline{\Cat}^-$.
\item The functors $\Ffun(\Delta_{\bmu,+}(\bullet))$ and $\Delta_{\bmu,-}(\bullet)$ are isomorphic.
\item The natural morphism $\Ffun\circ \Efun(M)\rightarrow M$ is an isomorphism for $M\in (\underline{\Cat}^-)^{\oDelta}$.
\end{enumerate}
\end{Lem}
\begin{proof}
Being a left adjoint of the exact functor $\Efun$, the functor $\Ffun$ maps projectives to projectives.
Since, thanks to Lemma \ref{Lem:Efun_prop}, $\Efun(L_-(\lambda))=L_+(\lambda)$, (1) follows.

To prove (2) we notice that the cokernel of $\Ffun\circ \Efun(M)\rightarrow M$ vanishes under $\Efun$.
Thanks to (3) of Lemma \ref{Lem:Efun_prop}, this implies that the cokernel is $0$.

Let us prove (3).  Similarly to the proof of (1) in Lemma \ref{Lem:Efun_prop}, it is enough to show that
the functors $\Ffun\circ\Delta_{\bmu,+}(\bullet)^{\oplus r!}=\iota^!\circ F_i^r\circ \pi^!\circ \Delta_{\bmu_+}(\bullet)$
and $\Delta_{\bmu,-}(\bullet)^{\oplus r!}=\Delta_{\bmu_+}(\otimes_{i=1}^{n-1}1_{\Cat^i_{\mu_i}}\otimes \,_n\!F_i^r(\bullet))$
are isomorphic. We have $\pi^!\circ\Delta_{\bmu,+}=\Delta_{\bmu}$. Then thanks to condition (TPC3) for $F_i$ applied
$r$ times, we see that $F_i^r\circ \Delta_{\bmu}(M)$ has a filtration with successive quotients of the form
$\Delta_{\bmu-\alpha}(F_i^{\alpha}M)$, each occurring with multiplicity $r!\binom{r}{m_1,\ldots,m_n}$
(our notation is the same as in the proof of Lemma \ref{Lem:Efun_prop}).
The only quotient lying in $\underline{\Cat}_-$ is the top quotient, it is naturally identified with
$\Delta_{\bmu,-}(\bullet)^{\oplus r!}$. The image of $\Delta_{\bmu-\alpha}(F_i^{\alpha}M)$ under $\iota^!$
coincides with the top quotient and so (3) is proved.

Let us prove (4). Thanks to (3) and (1) of Lemma \ref{Lem:Efun_prop}, the morphism $\Ffun\circ\Efun(\oDelta(\lambda))
\rightarrow \oDelta(\lambda)$ is an isomorphism. Now the proof repeats that of \cite[Lemma 5.3(2)]{LoHWCII}
(we remark that there the notation is different:  $\Ffun$ denotes an  exact functor, while $\Efun$ is its left adjoint).
\end{proof}

As in \cite{LoHWCII}, to show that $\Efun$ is an equivalence, it remains to prove the following result.

\begin{Lem}\label{Lem:Efun_equi}
$\Efun P_-(\lambda)=P_+(\lambda)$ for all $\lambda$, and $\Efun$ is fully faithful on $\underline{\Cat}^--\operatorname{proj}$.
\end{Lem}
\begin{proof}
We have an identification
$$\sigma: \Hom_{\underline{\Cat}^+}(P_+(\lambda),\Efun P_-(\lambda))=\Hom_{\underline{\Cat}^-}(\Ffun P_+(\lambda), P_-(\lambda))=
\End_{\underline{\Cat}^-}(P_-(\lambda),P_-(\lambda)).$$
We want to prove that $\varphi:=\sigma^{-1}(\operatorname{id})$ is an isomorphism. We have $\sigma(\varphi)=\eta\circ \Ffun\varphi\circ\theta$, where $\eta$ is a natural morphism $\Ffun\Efun P_-(\lambda)\rightarrow P_-(\lambda)$ that is an isomorphism
by (4) of Lemma \ref{Lem:Ffun_prop}, and $\theta$ is an isomorphism $\Ffun P_+(\lambda)\xrightarrow{\sim} P_-(\lambda)$ from (1)
of Lemma \ref{Lem:Ffun_prop}. We conclude that $\mathcal{F}\varphi=\eta^{-1}\circ \theta^{-1}$.

We claim that $P_+(\lambda)$ and $P_-(\lambda)$ have the same classes in $[\underline{\Cat}^+]=[\underline{\Cat}^-]$.
By Lemma \ref{Lem:BGG} to check this it suffices to show that $[\onabla_+(\lambda'):L_+(\lambda)]=[\onabla_-(\lambda'):L_-(\lambda)]$.
The latter follows from (2)-(4) of Lemma \ref{Lem:Efun_equi}. The same lemma now implies that the classes of $P_+(\lambda),\Efun P_-(\lambda)$ in $\underline{\Cat}^+$ coincide.  So, as in \cite[Lemma 5.4]{LoHWCII}, it is enough to prove that $\varphi$ is surjective.

Assume the converse, let $P_+(\lambda)\xrightarrow{\varphi} \Efun P_-(\lambda)$ have a nontrivial cokernel , say $K$.
Applying $\Ffun$ to the exact sequence $$P_+(\lambda)\xrightarrow{\varphi} \Efun P_-(\lambda)\rightarrow K\rightarrow 0$$
we get an exact sequence
$$\Ffun P_+(\lambda)\xrightarrow{\Ffun \varphi} \Ffun\Efun P_-(\lambda)\rightarrow \Ffun K\rightarrow 0.$$
But, being a composition of isomorphisms, $\Ffun\varphi$ is an isomorphism itself. So $\Ffun K=0$. It follows that
$\Hom_{\underline{\Cat}^-}(\Ffun K, L_-(\lambda'))=0$ for any $\lambda'$.
But the last Hom is $\Hom_{\underline{\Cat}^+}(K,\Efun L_-(\lambda'))=
\Hom_{\underline{\Cat}^+}(K, L_+(\lambda'))$. Since the latter is $0$ for all $\lambda'$ we deduce that $K$ is
zero. This completes the proof of $\Efun P_-(\lambda)=P_+(\lambda)$.

The full faithfulness follows from
\begin{align*}&\Hom_{\underline{\Cat}^-}(P_-(\lambda), P_-(\lambda'))=\Hom_{\underline{\Cat}^-}(\Ffun P_+(\lambda), P_-(\lambda'))=\\
&\Hom_{\underline{\Cat}^+}(P_+(\lambda), \Efun P_-(\lambda'))=\Hom_{\underline{\Cat}^+}(P_+(\lambda),P_+(\lambda')).\end{align*}
\end{proof}

\begin{Lem}\label{Lem:categor_commut}
We have $\Ffun\circ F_i\cong F_i\circ \Ffun$ and $\Efun\circ F_i\cong F_i\circ \Efun$.
\end{Lem}
\begin{proof}
Since we have already checked that $\Ffun,\Efun$ are mutually quasi-inverse equivalences, it is enough to show that
$\Efun\circ F_i\cong F_i\circ \Efun$. We have $\Efun\circ F_i=\pi\circ E_i^{(r)}\circ \iota \circ F_i\cong\pi\circ E_i^{(r)}F_i\circ\iota$
(because $F_i\circ \iota\cong \iota\circ F_i$) and $F_i\circ \Efun\cong \pi\circ F_i E_i^{(r)}\circ \iota$. As above, it is enough
to check that $\pi\circ E_i^rF_i\circ \iota\cong \pi\circ F_iE_i^r\circ\iota$. Also it is enough to do this blockwise.
But every block lies in a weight subcategory of $\Cat$. By \cite{CR04}, on a weight subcategory, we have an isomorphism
$F_i E_i^r\oplus (E_i^{r-1})^{\oplus d_1}\cong E_i^r F_i\oplus (E_i^{r-1})^{\oplus d_2}$, where $d_1,d_2$ are non-negative integers.
Repeating the argument of the proof of (1) in Lemma \ref{Lem:Efun_prop}, we see that $\pi\circ E_i^{r-1}\circ\iota=0$.
\end{proof}

\subsection{Functor \texorpdfstring{$\underline{E_i}$}{E_i}}
Recall that the subcategory $\underline{\Cat}^-\subset
\underline{\Cat}$ is closed under $F_i$, let $\underline{F}_i$ denote
the restriction of $F_i$ to $\underline{\Cat}^-$. The functor
$\underline{F}_i$ has both a left adjoint $F_i^!$, and a right adjoint
$F_i^*$. They are obtained as \[F_i^!=\iota^!\circ E_i\circ
\iota\qquad F_i^*=\iota^*\circ E_i\circ \iota,\] where $\iota$ is the
inclusion functor $\underline{\Cat}^-\hookrightarrow
\underline{\Cat}$, and $\iota^!,\iota^*$ are its left and right
adjoints.  To produce a functor $\underline{E}_i$ on
$\underline{\Cat}^-$ that, together with $\underline{F}_i$, will equip
$\underline{\Cat}^-$ with a categorical $\sl_2$-action, it is enough
to show that $F_i^!\cong F_i^*$.

We will approach this problem in a way analogous to \cite{LoHWCII}: we will show that both $F_i^!$ and $F_i^*$ will be isomorphic
to the third functor, $\underline{E}_i:=\Ffun\circ \pi\circ
E_i^{(r+1)}\circ \iota$.
\begin{Lem}\label{Lem:trunc_iso}
  We have isomorphisms of functors $F_i^!\cong \underline{E}_i\cong F_i^*$.
\end{Lem}
We note that this proof  closely follows \cite[5.3]{LoHWCII}.
\begin{proof}
If we prove the first equality, the second will follow by symmetry,
applying the theorem to $\Cat^{opp}$.

It is sufficient to prove that $\Efun\circ \underline{E}_i=\pi\circ E_i^{(r+1)}\circ \iota:\underline{\Cat}^-\rightarrow \underline{\Cat}^+$
is isomorphic to $\Efun\circ \underline{F}_i^!=\pi\circ E_i^{(r)}\circ\iota\iota^!\circ E_i\circ\iota$. Consider the adjunction epimorphism
$\operatorname{id}_{\underline{\Cat}}\twoheadrightarrow \iota\iota^!$. Composing it with $\pi\circ E_i^{(r)}$ on the left and
$E_i\circ\pi^!$ on the right, we get an epimorphism
\begin{equation}\label{eq:fun_epi_E_i}
\pi\circ E_i^{(r)}E_i\circ \iota\twoheadrightarrow \Efun\circ \underline{F}_i^!.
\end{equation}
Recall that we have a  decomposition $E_i^{(r)}E_i\cong \C^{r+1}\otimes E_i^{(r+1)}$. Picking a vector space embedding
$\C\hookrightarrow \C^{r+1}$, we get a functor morphism
\begin{equation}\label{eq:fun_mor}\pi\circ E_i^{(r+1)}\circ \iota\rightarrow \Efun\circ \underline{F}_i^!.\end{equation}
We need to check that, for each given weight subcategory, there is an embedding $\C\hookrightarrow \C^{r+1}$ that makes
the corresponding morphism (\ref{eq:fun_mor}) an isomorphism.

It is enough to check that there is an embedding $\C\hookrightarrow \C^{r+1}$ such that (\ref{eq:fun_mor})
is an isomorphism on all standardly filtered objects (because all functors under consideration are right exact).
This reduces to checking that (\ref{eq:fun_mor}) is an isomorphism on all objects $\Delta_{\bmu-}(N)$ in a given
weight subcategory as all functors under consideration are exact on standardly filtered objects (the embedding
$\C\hookrightarrow \C^{r+1}$ may depend on $\sum_{i=1}^{n-1}\mu_i$). This, in its turn, boils down to checking that there is an embedding $\C\hookrightarrow \C^{r+1}$ such that the composed morphism
\begin{equation}\label{eq:funct_cond2}\rho_{\bmu',+}\circ E_i^{(r+1)}\Delta_{\bmu,-}\hookrightarrow \rho_{\bmu',+}E_i^{(r)}E_i\Delta_{\bmu,-}\twoheadrightarrow \rho_{\bmu',+}E_i^{(r)}\circ F_i^!\Delta_{\bmu,-}\end{equation}
is an isomorphism of functors $\underline{\Cat}^-_{\bmu}\rightarrow \underline{\Cat}^+_{\bmu'}$.
Here $\rho_{\bmu',+}$ is the functor $\underline{\Cat}^{\oDelta}\rightarrow \Cat_{(\bmu',\nu_n)}$ that is the composition
of the left adjoint of the inclusion $\underline{\Cat}_{\leqslant (\bmu',\nu_n)}\hookrightarrow \underline{\Cat}$ and the projection $\underline{\Cat}_{\leqslant (\bmu',\nu_n)}\twoheadrightarrow
\Cat_{(\bmu',\nu_n)}$. The functor $\rho_{\bmu',+}$ is exact on $\Cat^{\oDelta}$. Therefore all functors in (\ref{eq:funct_cond2})
are exact.

From the weight considerations, all functors
in the exact sequence are $0$ unless $\bmu'=\bmu-\alpha_i^\ell$ for some $\ell=1,\ldots,n-1$. So consider $\bmu'$ of this form.
We claim that the rightmost functor is $\operatorname{id}^{\boxtimes\ell-1}\boxtimes \,_\ell\!E_i\boxtimes \operatorname{id}^{\boxtimes n-\ell-2}\boxtimes \,_n\!E_i^{(r)}$. Indeed, since $\Efun:\underline{\Cat}^-\rightarrow \underline{\Cat}^+$ is an equivalence
of standardly stratified categories, we see that the right-most functor equals $\Efun\circ \rho_{\bmu',-}\circ F_i^!\circ\Delta_{\bmu,-}$.
But $\rho_{\bmu',-}\circ F_i^!\circ\Delta_\bmu=\rho_{\bmu',-}\circ E_i\circ\Delta_\bmu=\operatorname{id}^{\ell-1}\boxtimes \,_{\ell}\!E_i\boxtimes\operatorname{id}^{n-\ell-1}$. Our claim follows since the equivalence $\underline{\Cat}^-_{\bmu',-}\rightarrow
\underline{\Cat}^+_{\bmu',+}$ is $\operatorname{id}^{\boxtimes n-1}\boxtimes \,_n\!E_i^{(r)}$. Therefore the right functor
maps a simple object to an object with simple head.

Also the middle functor in (\ref{eq:funct_cond2}) is the sum of $r+1$ copies of the left-most functor. Apply the functors
in (\ref{eq:funct_cond2}) to a simple object $N$. The object on the right has simple head. It follows that the set
of embeddings $\C\hookrightarrow \C^{r+1}$ such that the composition in (\ref{eq:funct_cond2}) is surjective on $N$
is a complement to a hyperplane in $\mathbb{P}^r$. The number of simples in $\Cat_{\bmu,-}$ with given $|\bmu|$ is finite
so we have an open subset of $\mathbb{P}^r$ such that the composition is surjective if we choose our embedding $\C\hookrightarrow\C^{r+1}$
in this subset. But using condition (TPC3) in the definition of a tensor product categorification one sees that $$[\rho_{\bmu',+}E_i^{r+1}\Delta_{\bmu,-}(N)]=(r+1)![\operatorname{id}^{\boxtimes\ell-1}\boxtimes \,_{\ell}\!E_i\boxtimes \operatorname{id}^{\boxtimes n-\ell-2}\boxtimes \,_{n}\!E_i^{(r)}(N)].$$
It follows that the classes of $\rho_{\bmu',+}E_i^{(r+1)}\Delta_{\bmu,-}(N)$ and
$\rho_{\bmu',+}E_i^{(r)}\circ F_i^!\Delta_{\bmu,-}(N)$ in the Grothendieck group coincide.
So any epimorphism between the two objects has to be an isomorphism.
This completes the proof.
\end{proof}

Using the identification $\Efun:\underline{\Cat}^-\hookrightarrow
\underline{\Cat}^+$, we can transfer $\underline{E}_i$ to
$\underline{\Cat}^+$. Let  $\underline{F}_i\colon
\underline{\Cat}^+\to \underline{\Cat}^+$ denote the functor
induced by $F_i$.
Thanks to Lemma \ref{Lem:categor_commut}, we see that, being both left and right adjoint to $\underline{F}_i$, the functor $\underline{E}_i$
preserves the subcategories $(\underline{\Cat}^{\pm})^{\oDelta},(\underline{\Cat}^{\pm})^{\Delta}, (\underline{\Cat}^{\pm})^{\onabla},
(\underline{\Cat}^{\pm})^{\nabla}$.

\begin{Lem}\label{Lem:fun_coinc}
Under the embedding
$\pi^!:(\underline{\Cat}^{+})^{\oDelta}\hookrightarrow
\Cat^{\oDelta}$, we have an isomorphism of functors $\pi^!\circ
\underline{E}_i\cong E_i\circ \pi^!$. The induced isomorphism $\pi^!\circ
\underline{E}_i^n\cong E_i^n\circ \pi^!$ intertwines the $R$ actions
on these functors.
\end{Lem}
\begin{proof}
It follows from condition (TPC3) that $E_i$ preserves
$\pi^!(\underline{\Cat}^{+})^{\oDelta}$, so it induces a functor
$E'_i\colon (\underline{\Cat}^{+})^{\oDelta}\to
(\underline{\Cat}^{+})^{\oDelta}$ such that $\pi^!\circ E'_i\cong
E_i\circ \pi^!$.  In fact, this functor can be described as $E'_i\cong
\pi\circ E_i\circ \pi^!$.  The right adjoint of this functor is
$\pi\circ F_i\circ \pi^*\cong \underline{F}_i$.  Taking the left adjoint of
this isomorphism, we obtain an isomorphism $E'_i\cong
\underline{E}_i$.  Since the action of $R$
on $\underline{F}^n$ is by definition induced by that  on $\pi\circ F^n\circ \pi^*$, the adjoint
isomorphism also intertwines these actions on $\pi\circ E^n\circ
\pi^!$ and $\underline{E}^n$.
\end{proof}

\subsection{Checking conditions}

\begin{Thm}
  The functors $\underline{E}_i$ and $\underline{F}_i$ give rise to a categorical action on
  $\underline{\Cat}^+$ inducing the tensor product action on $[\underline{\Cat}^+]\cong V_1\otimes\cdots \otimes V_{n-1}$.
\end{Thm}
\begin{proof}
  Here we apply \cite[5.27]{Rou2KM}:
  \begin{itemize}
  \item We have already checked that $\underline{E}_i$ and
    $\underline{F}_i$ are adjoint.
  \item The functors $\underline{E}^k$ inherit a KLR action from
    $\Cat$.
  \item It's clear that these functors change weights in the correct way.
  \end{itemize}
  Thus, we need only check that the linear maps $[\underline{E}_i]$
  and $[\underline{F}_i]$ induce an integrable action of $\g$ on
  $[\underline{\Cat}^+]\cong V_1\otimes\cdots \otimes V_{n-1}$.  Of
  course, we can easily check that they act with the usual tensor
  product action; this follows for $[E_i]$ by Lemma
  \ref{Lem:fun_coinc}, and for the $[F_i]$, one simply notes that
  $\pi^!(\underline{F}_i\Delta_+(\bla))$ is the kernel of the natural map
  $F_i\Delta(\bla)\to \Delta({}_nF_iQ(\bla))$; this kernel is, of
  course, filtered by $\Delta({}_jF_iQ(\bla))=\pi^!\Delta({}_j\underline{F}_iQ(\bla))$.  This precisely shows that
\[[\underline{F}_i](v_1\otimes \cdots v_{n-1})=\sum_{j=1}^{n-1}v_1\otimes
\cdots \otimes [{}_j\underline{F}_i]v_j\otimes\cdots\otimes v_{n-1}.\]
Thus, we are done.
\end{proof}

\begin{Cor}
  With its induced categorical action, $\underline{\Cat}^+$ is a tensor product categorification of
  $V_1\otimes\cdots\otimes V_{n-1}$.
\end{Cor}
\begin{proof}
  Conditions (TPC1-2) are straightforward from the construction.
  Condition (TPC3) for $F_i$ follows directly from the construction of
  $F_i$ as an induced functor. Condition (TPC3) for $E_i$ follows from
   Lemma \ref{Lem:fun_coinc}.
\end{proof}

There is also a  ``dual" splitting that will be used below. While our original splitting is designed to be compatible
with projectives, the dual one is rather compatible with tiltings.

The category which is splitting off is the subcategory $\bar{\Cat}^-\subset \Cat$ spanned by all simples $L(\lambda)$
with $\varrho(\lambda)=(\nu_1,\ldots)$. The set $\{(\nu_1,\ldots)\}$ is a poset ideal and so
$\bar{\Cat}^-\subset \Cat$ is a standardly stratified subcategory. This subcategory is stable with respect
to the functors $\underline{E}_i:=E_i$ but not stable with respect to $F_i$. However one can truncate the functors
$F_i$ getting the endofunctors $\underline{F}_i$ completely analogously to the above (first constructing the equivalence
$\mathcal{F}$ of $\bar{\Cat}^-$ with a suitable subquotient of $\Cat$ and then showing that $\underline{E}_i^!\cong \underline{E}_i^*$
by analogy with Lemma \ref{Lem:trunc_iso}; we remark that we do not need an analog of Lemma \ref{Lem:fun_coinc}).
With these functors, $\bar{\Cat}_-$ becomes a tensor product categorification of
$V_2\otimes\ldots\otimes V_n$.

\excise{
\begin{Rem}
Another way to see this categorification on $\bar{\Cat}^-$ is to
conjugate our primary construction  by
Ringel duality: we have
$\bar{\Cat}^-=\,^\vee\!(\underline{\Cat}^{\vee+})$, where we write $\,^\vee\!\bullet:=[(\bullet^{opp})^{\vee}]^{opp}$.
  The categorification
functors $E_i$ on $\,^\vee\!(\underline{\Cat}^{\vee+})$ coincide with
the restrictions of $E_i$ to $\bar{\Cat}^-$; this can be seen from Lemma \ref{Lem:fun_coinc} (applied to $\Cat^{\vee opp}$).
\end{Rem}
}
\section{Double centralizer property}
\subsection{Statement}
Let $\Cat$ be a tensor product categorification of $V_1\otimes\cdots\otimes V_n$. As before let $\bnu=(\nu_1,\ldots,\nu_n)$ be the sequence
of highest weights of $V_1,\ldots,V_n$,
and we write $|\nu|$ for $\sum_{i=1}^n\nu_i$. We consider the projectives in $\Cat$ that are direct summands of
$F^N \mathbb{V}$ for  $N\in \Z_{\geqslant 0}$, where $\mathbb{V}$ was
defined in (TPC2).
The corresponding quotient functor $\pi_{top}$ kills all simples $L(\lambda)$
such that $|\bnu-\varrho(\lambda)|$ is the sum of $k$ simple roots but $E^k L(\lambda)=0$. The quotient
category categorifies the Cartan irreducible component of $V_1\otimes\cdots\otimes V_n$
(=the only irreducible component with highest weight $|\bnu|$); we
denote it by $\Cat_{top}$.
The category $\Cat_{top}$ has an induced categorical action which
makes $\pi_{top}$ strongly equivariant.
By Proposition \ref{simple-unique},
$\Cat_{top}$ is strongly equivariantly
equivalent to the representations of a cyclotomic quotient of
the KLR algebra $R^{|\bnu|}\mmod$. As shown in the proof of that
theorem, this functor can be identified
with $M\mapsto \Hom(\mathbb{V},E^kM)=\Hom(F^k \mathbb{V},M)$; this has
a canonical action of the KLR algebra which factors through the
cyclotomic quotient $R^{|\nu|}$.

We can canonically identify $\Cat_{top}\operatorname{-proj}\cong \Cat_{top}\operatorname{-inj}$ with the
subcategory of $\Cat$ additively generated by the objects $F^N
\mathbb{V}$.

Here is the main result of this section.

\begin{Thm}\label{Thm:doub_centr}
The functor $\pi_{top}$ is fully faithful on projectives.
\end{Thm}

\excise{Our proof follows the approach of \cite{GGOR} to the KZ functor for rational Cherednik algebras, compare
with \cite{Lotowards}.}
Our approach closely follows  Soergel's proof of the Struktursatz \cite{Soe90}.
An essential prerequisite for this proof is to check that every simple appearing in the socle of an object from $\Cat^{\oDelta}$ survives under $\pi_{top}$.

\subsection{Socles of standard objects}
\begin{Prop}\label{Prop:stand_soc}
Pick $\bmu$ such that $|\bnu-\bmu|$ is the sum of $k$ simple roots. If a simple $L$ appears in the socle
of $\Delta_\bmu(N)$ with $N\in \Cat_\bmu$, then $E^k L\neq 0$.
\end{Prop}
\begin{proof}
The claim is vacuous when $n=1$. So in the proof we can assume that the claim is proved for all tensor product
categorifications of products with $n-1$ factors, in particular, for
$
{\bar{\Cat}}^-\subset \Cat$.
Since the subcategory $
{\bar{\Cat}}^-$ is closed under $E$,
this establishes our claim when $\mu_1=\nu_1$. From
now on we may assume that $\mu_1<\nu_1$.

Now we prove our claim by induction on  the number $\ell$ of simple roots in the decomposition of $\nu_1-\mu_1$.
We may assume that $N$ is simple, $N=N_1\boxtimes\cdots\boxtimes N_n$.
Since $\mu_1<\nu_1$, we see that $N_1$ occurs in the socle (equivalently, top) of $F_i N_1'$ for some $i$
and some simple $N_1'\in \Cat^1_{\mu_1+\alpha_i}$. Set $N'=N_1'\boxtimes N_2\boxtimes\ldots\boxtimes N$
and consider the object $F_i\Delta_{\bmu+\alpha_i^1}(N')$. According to condition (TPC3), this object has a filtration
with subobject $\Delta_\bmu({}_1F_i N')$.  This induces an injection
$\Delta_\bmu(N)\hookrightarrow F_i\Delta_{\bmu+\alpha_i^1}(N')$, and
thus an injection \[\Hom(L,\Delta_\bmu(N))\hookrightarrow \Hom(L, F_i\Delta_{\bmu+\alpha_i^1}(N'))=
\Hom(E_i L, \Delta_{\bmu+\alpha_i^1}(N')).\] Thus, if the former space
is nonzero, the latter is as well. Any non-zero map $E_iL\to
\Delta_{\bmu+\alpha_i^1}(N')$ induces an injection of a simple composition
factor $L'$ of $E_i L$ into the socle of
$\Delta_{\bmu+\alpha_i^1}(N')$.  By the inductive assumption,
$E^{k-1}L'\neq 0$ and hence $E^k L\neq 0$ by the exactness of $E$.
\end{proof}

\begin{Lem}\label{lem:standard-injective}
  For any $M\in \Cat^{\Delta}$, we have an injection
  $M\hookrightarrow P$ with $P\in \Cat_{top}\operatorname{-proj}$. Moreover, if $M\in \Cat\operatorname{-proj}$,
  then we can choose this map so that
  $P/M\in \Cat^{\Delta}$.
\end{Lem}
\begin{proof}
By Proposition \ref{Prop:stand_soc}, the injective hull of $M\in \Cat^{\Delta}$ lies in $\Cat_{top}\operatorname{-inj}=
\Cat_{top}\operatorname{-proj}$, which shows that the desired
injection exists.
For the remainder of the proof, we assume that $M\in \Cat\operatorname{-proj}$.

We prove this by induction on $k$, where $k$ has the same meaning as in Proposition
\ref{Prop:stand_soc}. If $k=0$, then this is trivial.
Now, fix $k$, and assume the statement holds for $k-1$.  Then for any
$M\in \Cat\operatorname{-proj}$, we have that $EM\in \Cat\operatorname{-proj}$, so by
induction, we have a map $EM\hookrightarrow P'$ for $P'\in
\Cat_{top}\operatorname{-proj}$ and $P'/EM\in \Cat^{\Delta}$. This gives rise
to an embedding $FEM\hookrightarrow FP'$ and we have $FP'/FEM=F(P'/EM)\in \Cat^{\Delta}$
by (TPC3). So it remains to show that there is an embedding $M\hookrightarrow FEM$
with $FEM/M\in \Cat^\Delta$. We take the morphism $M\rightarrow FEM$ obtained from the
identity morphism $EM\rightarrow EM$ by adjunction. The induced morphism $EM\rightarrow EFEM$
is an embedding. Since there is no simple in the socle of $M$ killed by $E$, we deduce that
the morphism $M\rightarrow FEM$ is an embedding.
\excise{ Thus,
we have an induced map $M\to FP'$. If this map is not injective then
the induced map $EM\to EFP$ is not injective; however
$EM\hookrightarrow P'$ and $P'\to EFP'$ are both injective, so this is
impossible.  Thus, this map is injective, and we need only check that
$FP'/M\in \Cat^{\oDelta}$.  We have an injective map $FEM/M\to FP'/M$ and so
it suffices to check that $FP'/FEM \cong F(P'/EM)\in \Cat^{\Delta}$
and $EFM/M \in \Cat^{\Delta}$.  In the first case, this follows by
(TPC3).}

By Lemma \ref{Lem:cost_filt}(2), it remains to check that $\Ext^1(FEM/M ,\onabla(\bla))=0$ for all
$\bla$.  From the usual long exact sequence, this is equivalent to the
surjectivity of the induced map \[\Hom(FEM,\onabla(\bla))\to\Hom(M,
\onabla(\bla))\] for the unit $M\to FEM$ of the adjunction $(E,F)$;  by the
biadjunction of $F$ and $E$, this is in turn equivalent to the
surjectivity of the map \[\Hom(M,FE\onabla(\bla))\to\Hom(M,
\onabla(\bla))\] induced by the surjective {\it co}unit map
$FE\onabla(\bla)\to \onabla(\bla)$ for the adjunction $(F,E)$.  Of
course, this is just the universal property of projectives.
\end{proof}

\begin{proof}[Proof of Theorem \ref{Thm:doub_centr}]
It is enough to check that $\Hom_{\Cat}(M,M')\xrightarrow{\sim}\Hom_{\Cat^{top}}(\pi_{top}M,\pi_{top}M')$
when $M'\in \Cat \operatorname{-proj}$. By Lemma \ref{lem:standard-injective},
we have a short exact sequence \[0\to M'\to P_1\to P_2\] for
$P_1,P_2\in \Cat_{top}\operatorname{-proj}$.  We thus have a short
exact sequence \[0\to \pi_{top}(M')\to \pi_{top}(P_1)\to
\pi_{top}(P_2),\] with the latter two terms injective modules.
Essentially by definition, we have that $\Hom_{\Cat}(M,P_i)\cong
\Hom_{\Cat^{top}}(\pi_{top}M,\pi_{top}P_i)$ for any module $M$.
Thus, we have the diagram with exact rows
\[
\tikz[->,very thick]{
\matrix[row sep=15mm,column sep=9mm,ampersand replacement=\&]{
\node (a) {$0$}; \& \node (c) {$\Hom_{\Cat}(M,M')$}; \& \node (e)
{$\Hom_{\Cat}(M,P_1)$};\& \node (g) {$\Hom_{\Cat}(M,P_2)$};\\
\node (b) {$0$}; \& \node (d) {$\Hom_{\Cat^{top}}(\pi_{top}M,\pi_{top}M')$};\& \node (f) {$\Hom_{\Cat^{top}}(\pi_{top}M,\pi_{top}P_1)$};\& \node (h) {$\Hom_{\Cat^{top}}(\pi_{top}M,\pi_{top}P_2)$};\\
};
\draw (c) -- (e);
\draw (e) -- (g);
\draw (a) -- (c);
\draw (b) --(d);
\draw (d) --(f);
\draw (f) --(h);
\draw (c)--(d) node[right,midway]{$(*)$};
\draw (e)--(f) node[right,midway]{$\sim$};
\draw (g)--(h) node[right,midway]{$\sim$};
}
\]
By the 5-lemma, the map marked $(*)$ is an isomorphism.
\end{proof}

\section{Proof of the uniqueness theorem}
\subsection{Main result}

We intend to give a classification of all tensor product
categorifications.

\begin{Thm}\label{main2}
Let $\Cat$ be a tensor product categorification of
$V_1\otimes\dots\otimes V_n$. Then we have a strongly equivariant
equivalence $\Cat\cong \Cat(\bnu)$ of standardly stratified categories
that preserves the labels of simples.
\end{Thm}

Consider the categorification $\underline{\Cat}^+$; since the highest
weight of this categorification is $|\nu|-\nu_n$, we have a functor
$\underline{\pi}_{top}$ from $\underline{\Cat}^+$ to the category
$R^{|\nu|-\nu_n}\mmod$.  We wish to compare this to the functor $\pi_{top}$.

\begin{Lem}\label{lem:top-inflate}
  We have a commutative diagram
  \begin{equation*}
    \tikz[->,scale=2, very thick]{
\node (a) at (0,0) {$(\underline{\Cat}^+)^{\bar\Delta}$};
\node (b) at (2,0) {$\Cat^{\bar\Delta}$};
\node (c) at  (0,1) {$R^{|\nu|-\nu_n}\mmod$};
\node (d) at (2,1) {$R^{|\nu|}\mmod$};
\draw (a)-- (b) node[below,midway]{$\pi^!$};
\draw (a)-- (c) node [left,midway]{$\underline{\pi}_{top}$};
\draw (b)-- (d) node [right,midway]{$\pi_{top}$};
\draw (c)-- (d) node [above, midway]{$\operatorname{inf}$};
}
  \end{equation*}
\end{Lem}
\begin{proof}
  By Lemma \ref{Lem:fun_coinc}, the vector spaces $E^k\pi^!(M)$ and
  $\underline{E}^k(M)$ are naturally isomorphic as $R$-modules.
  Thus, if we think of $E^k\pi^!(M)$ as an
  $R^{|\bnu|}$-module, it is simply obtained by considering
  $\underline{E}^k(M)$ as an $R^{|\bnu|-\nu_n}$-module and pulling back.
\end{proof}

\begin{proof}[Proof of Theorem \ref{main2}]
Let $\pi_1,\pi_2$ denote the quotient functors $\Cat\twoheadrightarrow \underline{\Cat}^+$
and $\Cat(\nu_1,\ldots,\nu_n)\twoheadrightarrow \Cat(\nu_1,\ldots,\nu_{n-1})$. Further
let $\pi^1_{top},\pi^2_{top}$ be the quotient functors from $\Cat$ and $\Cat(\nu_1,\ldots,\nu_n)$
to $R^{|\bnu|}\operatorname{-mod}$ and $\underline{\pi}^1_{top},\underline{\pi}^2_{top}$ the quotient functors
from $\underline{\Cat}^+,\Cat(\nu_1,\ldots,\nu_{n-1})$ to
$R^{|\bnu|-\nu_n}\mod$. Below we will sometimes write $\Cat_1$ for $\Cat$, and $\Cat_2$ for $\Cat(\nu_1,\ldots,\nu_n)$
(and also $\underline{\Cat}_1^+$ for $\underline{\Cat}^+$ and $\underline{\Cat}_2^+$ for $\Cat(\nu_1,\ldots,\nu_{n-1})$).

First, note that it suffices to check this theorem on the categories
of projective objects in each category.  By Theorem
\ref{Thm:doub_centr}, we can strongly equivariantly identify
$\Cat\operatorname{-proj}$ and $\Cat(\nu_1,\ldots,\nu_n)
\operatorname{-proj}$ with their images under $\pi^i_{top}$.  Thus,
the desired equivalence would follow from showing that
$\{\pi^1_{top}P^1(\bla)\}= \{\pi^2_{top}P^2(\bla)\}$, where
$P^1(\bla)$ and $P^2(\bla)$ are the indecomposable projectives
corresponding to $\lambda$ in $\Cat$ and
$\Cat(\nu_1,\ldots,\nu_n)$. To show that the equivalence is of
stratified categories and preserves the labels of simples, it is enough
to check that $\pi^1_{top}P^1(\bla)=\pi^2_{top}P^2(\bla)$ (indeed, the
only additional structure on a standardly stratified category is a
pre-order on the set of simples so if an equivalence preserves the
labels, then it will automatically intertwine the standardly
stratified structures).

Our proof proceeds by induction.  If $n=1$, then $\Cat\cong
\gr\Cat$, and the conclusion follows from Proposition \ref{simple-unique}.

For arbitrary $n$, we can conclude by the inductive hypothesis that
$\underline{\Cat}^+\cong \Cat(\nu_1,\dots, \nu_{n-1})$ and the equivalence
preserves the labels of the simples. Such an equivalence automatically intertwines
the functors $\underline{\pi}^1_{top}$ and $\underline{\pi}^2_{top}$. 
Thus, by Lemma \ref{lem:top-inflate}, $\pi^1_{top}\circ\pi_1^!(P_+^1(\bla))=
\pi^2_{top}\circ \pi_2^!(P^2_+(\bla))$ for all labels $\bla$ in
$\underline{\Cat}^+\cong \Cat(\nu_1,\ldots,\nu_{n-1})$. Since $\pi_i^! P^i_+(\bla)=
P^i(\bla)$, we see that the objects $\pi^i_{top}P^i(\bla)$ coincides when
$\bla=(\lambda_1,\ldots,\lambda_n)$ with $\varrho^n(\lambda_n)=\nu_n$.

We claim that any projective $P^i(\bla')$ appears as a summand in $F^N P^i(\bla)$.
Indeed, the classes of projectives $P^i(\bla)$ with $\varrho^n(\lambda_n)=\nu_n$ are a basis
in $V_1\otimes\ldots\otimes V_{n-1}\otimes v_{\nu_n}\subset V_1\otimes\ldots\otimes V_n$.
Since $V_1\otimes\ldots\otimes V_{n-1}\otimes v_{\nu_n}$ generates the $U(\mathfrak{n}^-)$-module $V_1\otimes\ldots\otimes V_{n}$,
we see that the classes of summands of $F^N P^i(\bla)$ (for all $N$
and $\bla$) generate $V_1\otimes\ldots\otimes V_n$, so there can be no
other projectives.

So any $\pi^i_{top}P^i(\bla')$ is an indecomposable (thanks to Theorem \ref{Thm:doub_centr})
summand of $F^N \pi^i_{top} P^i(\bla)$ with $\varrho^n(\lambda_n)=\nu_n$. It follows that
the sets $\{\pi^i_{top}P^i(\lambda)\}$ coincide, which proves the existence of a strongly
equivariant equivalence
$\beta:\Cat\rightarrow\Cat(\nu_1,\ldots,\nu_n)$.

Now, we need only check that the labels match. Let us choose a sequence of nodes $i_1,i_2,\dots$ as in
Section \ref{sec:labeling}, and order infinite sequences of
non-negative integers almost all of which are 0 by:
\begin{itemize}
\item If $\sum a_i < \sum b_i$, then $\Ba=(a_1,\dots) > \Bb=(b_1,\dots)$.
\item If $\sum a_i = \sum b_i$, then we use lexicographic order.
\end{itemize}
For each indecomposable projective object $P^i(\bla)$, where $i=1,2$,
there is a unique word $\Ba=(a_1,\dots)$ maximal in this order for which
$P^i(\bla)$ appears in $ F_{i_1}^{a_1}F_{i_2}^{a_2}\cdots Q$ for $Q$
projective in $\Cat^+$. This is equivalent to $\Ba$ being maximal so
that $Q$ has a non-zero map to  $\cdots
E_{i_2}^{a_2}E_{i_1}^{a_1}L^i(\bla)$.  In particular, the standard cover of
some composition factor of $\cdots
E_{i_2}^{a_2}E_{i_1}^{a_1}L^i(\bla)$
must have a label whose
$n$th component is $\upsilon_n$, the only indecomposable object in $\Cat^n_{\nu_n}$.
We call such a label a {\bf plus-label}.  Note that
a simple with plus-label can only occur as a composition factor of a
standard with plus-label.

Let $\Bb$ be the string parameterization, see Section \ref{sec:labeling}, of
$\lambda_n$. Assume $\Ba$ is larger than $\Bb$ in the order above.


Now, consider $\cdots E_{i_2}^{a_2}E_{i_1}^{a_1}L^i(\bla)$. This is a quotient of $\cdots
E_{i_2}^{a_2}E_{i_1}^{a_1}\bar\Delta^i(\bla)$; this module has a canonical
filtration by proper standard modules.  If $\Ba$ has smaller sum than $\Bb$,
then none of these standards has plus-label, since even the successive
quotient where we only use $E$'s from the last component cannot have
high enough weight in the last label.  On the other hand, if $\Ba$ has
the same sum, but is higher in lexicographic order, then the component
of the standard filtration where we apply all $E$'s in the last
component will have the correct weight, but be trivial by the
definition of string parametrization.
Thus, the word $\Ba$ mentioned above must be no greater in our order
than $\Bb$.

On the other hand, we claim that the object $\cdots E_{i_2}^{b_2}E_{i_1}^{b_1}L^i(\bla)$ contains the simple
with label $\bla^+$ matching $\bla$ except with $\upsilon_n$ in the last component as a
composition factor. Indeed, by (TPC3), $\oDelta^i(\bla^+)$ is a composition factor
of  $\cdots E_{i_2}^{b_2}E_{i_1}^{b_1}\oDelta^i(\bla)$ and the label $\bla^+$ is the largest among
the labels of the composition factors. If $L^i(\bla^+)$ appears in the kernel of the projection
$\cdots E_{i_2}^{b_2}E_{i_1}^{b_1}\oDelta^i(\bla)\twoheadrightarrow \cdots E_{i_2}^{b_2}E_{i_1}^{b_1}L^i(\bla)$,
then it also appears in $\cdots E_{i_2}^{b_2}E_{i_1}^{b_1}\oDelta^i(\bla')$ for some $\bla'<\bla$.
This is impossible if $\lambda'_n<\lambda_n$. Furthermore, if $\lambda'_n=\lambda_n$, then the largest
label of a composition factor in  $\cdots E_{i_2}^{b_2}E_{i_1}^{b_1}\oDelta^i(\bla')$ is $\bla'^+<\bla^+$.
This contradiction shows the claim in the beginning of the paragraph that, in turn, implies $a_i=b_i$
for all $i$.

Note that  all other proper standards appearing in $\cdots
E_{i_2}^{a_2}E_{i_1}^{a_1}\bar \Delta^i(\bla)$ do not have plus-labels, and
thus have no simple composition factors with plus-labels. Thus, if any other simple with a plus label occurs as a
composition factor $\cdots E_{i_2}^{b_2}E_{i_1}^{b_1}L^i(\bla)$, then is must a composition factor of
$\bar \Delta(\bla^+)$.
Thus, every other plus-label attached to a simple in $\cdots
E_{i_2}^{b_2}E_{i_1}^{b_1}L^i(\bla)$ must be $<\bla^+$.

This shows the uniqueness of labels: the word $\Ba$ is
determined by the definition of the action and the category $\Cat^+$,
and so must match for $P$ and $\beta(P)$; this shows that the last
term of the labels match, and the label $\bla^+$ is distinguished as
the maximal plus-label for a composition factor in $\cdots E_{i_2}^{a_2}E_{i_1}^{a_1}L^i(\bla)$. Since we already know that
the labels match for $\Cat^+$ by induction, this establishes the
general case.
\end{proof}

\subsection{Consequences}
\label{sec:consequences}

This theorem shows that many structures on tensor product
categorifications thus come for free:
\begin{Cor}\label{grading}
  If the polynomials $Q_{ij}(u,v)$ are chosen to be homogeneous, then any tensor product categorification has a unique
  graded lift (given by graded modules of $T^\bnu$) which carries a
  homogeneous action of $\mathfrak{A}$; that is, between any two such
  lifts, there is a strongly equivariant graded lift of the identity functor which is
  unique up to unique isomorphism.
\end{Cor}
\begin{proof}
  The existence of this lift is clear, so we turn to its uniqueness.
  The natural map $R_k\to \End(F^k\mathbb{V})$ is surjective with
  homogeneous kernel, and thus induces a grading on the latter space.
  Furthermore, every indecomposable projective has an injective map
  into $F^k\mathbb{V}^{\oplus p}$ for some $p$ (its injective hull is
  a summand of this module).  We call a projective submodule $P\subset
  F^k\mathbb{V}^{\oplus p}$ {\bf homogeneous} if the left ideal of
  $\End(F^k\mathbb{V}^{\oplus p})$ consisting of endomorphisms whose
  image lies in $P$ is homogeneous for this grading.  Note that this module
  coincides with
  $\Hom(F^k\mathbb{V}^{\oplus p}, P)$ and thus grades this space.  By the double centralizer
  property, this induces a grading on the Hom space between any two
  homogeneous projective submodules of $F^k\mathbb{V}^{\oplus p}$.

The category of homogeneous projective submodules of
$F^k\mathbb{V}^{\oplus p}$ is thus a graded lift of the category of
projectives in $\mathcal{C}(\bnu)$, which only depends on the choice
of grading on $R_k$.  Any other graded lift has a canonical functor
from its category of homogeneous projectives to this lift induced by
the functor $\pi_{top}$, so this establishes the desired uniqueness.
\end{proof}

\begin{Cor}[\mbox{\cite[Thm. \ref{m-categor}]{Webmerged}}]
  The tensor product categorifications for different orderings of the
  same representations have  equivalent derived categories.
\end{Cor}
As mentioned before, there are different, competing notions of
categorical $\fg$-action;  the most obvious variation of the
definition we have used is the 2-category of Cautis and Lauda, where
rather than simply adjoining an inverse of $\rho_{s,\la}$, the
relation that another morphism is its inverse is imposed.  The
difference between these definitions is subtle, and the evidence thus
far suggests that most interesting actions in the sense of Rouquier
can be strengthened to one of these.  For example, in our case, we find:
\begin{Cor}[\mbox{\cite[Thm. \ref{m-main}]{Webmerged}}]
  Any tensor product action can be strengthened to an action of the
  2-category of
  Cautis and Lauda \cite{CaLa}.
\end{Cor}

Moreover, Theorem \ref{main2} provides a new and independent proof of
\cite[\ref{m-equiv}]{Webmerged}, which establishes an equivalence of
$\Cat$, the sum of blocks in a parabolic category $\OCat$ mentioned in
Section \ref{sec:relat-with-prev}, and the category
$\Cat(\omega_{m_1},\omega_{m_2},\ldots,\omega_{m_n})$.  It also shows that the subquotient categories $\Cat_e$ from
Section \ref{sec:relat-with-prev} only depend on the ordered products
they categorify and not on the equivalence class $e$.

These techniques are also useful for analyzing the category $\cO$ for
a Lie superalgebra $\mathfrak{gl}(m|n)$.  These applications will be
explored further is a forthcoming joint paper of Brundan and the
authors \cite{BLW}.

\section{Crystals}
\subsection{Main result}

Recall that we have defined crystal operators $\tilde{e}_i,\tilde{f}_i$ on the set of simples
in a tensor product categorification $\Cat$ of
$V_1\otimes\cdots\otimes V_n$; this induces a crystal structure on
$\Lambda$, which we call the {\bf categorical} crystal structure to
avoid confusion.

Our main result in this section is to give a combinatorial description
of this structure; both the result and its proof generalize those from
\cite{LoHWCI}. Recall that the set of labels for simples in $\Cat$ is
$\{\bla=(\lambda_1,\ldots,\lambda_n)\}$, where $\lambda_j$ is a label
of a simple in the categorification $\Cat^j$ of $V_j$.  Recall that
the crystal of $\Cat^j$ is known: it is isomorphic to the crystal of
$V_j$ by \cite[\S 5.1]{LV}. Since this crystal is irreducible, this
determines the crystal operators,
${}_{j}\tilde{e}_i,{}_{j}\tilde{f}_i$, uniquely.

To describe the crystal operators for $\Cat$, we will need some notation
and the notion of the $i$-signature.  Let
\begin{itemize}
  \item $h_+^j(\lambda_j)$ be the maximum over all integers $k$ such that
  $({}_{j}F_i)^k L^j(\lambda_j)\neq 0$, and
\item $h_-^j(\lambda_j)$ be
  maximum over all integers $\ell$ such that $({}_{j}E_i)^\ell
  L^j(\lambda_j)\neq 0$.
\end{itemize}
\begin{defi}
  For each index $j$, we have a sign sequence given by concatenating $h_+^j(\lambda_j)$ many  $+$-signs,
  followed by $h_-^j(\lambda_j)$ many  $-$-signs.
  The {\bf $i$-signature} of $\bla$ is a sequence of $+$'s and $-$'s
  given by concatenating the sequences from each index in turn.
\end{defi}

For example, consider the case where $\g=\sl_2$, and
$\nu_1=\nu_2=\nu_3=3$.  In this case all $i$-signatures will be of
length 9; for example, if $\la_1$ is the  unique crystal element of
weight $1$ and $\la_2,\la_3$ have weight $-1$, then we have three groups: $(++-), (+--),(+--)$,
then the $i$-signature is $(++-+--+--)$.  Note that outside $\sl_2$,
the length of the $i$-signature can differ for labels in the
same representation.

We annotate each $i$-signature as follows: if you find a consecutive pair
of the form $-+$, cross both out, and continue this process, ignoring
crossed out symbols, until every pair of uncrossed $+$ and $-$-signs
have the $+$ sign to the left.  This process is often visualized by
replacing $+$ with the open parenthesis symbol $($ and $-$ with the
close parenthesis $)$ and ignoring matching
parentheses; in \cite[2.4]{LoHWCI}, symbols are turned into $0$'s
rather than struck out.  This allows us to speak, in any $i$-signature, of crossed
and uncrossed symbols. We let $h_{\pm}(\lambda)$ be the number of
uncrossed symbols remaining.  In the example above, we strike out the
symbols in positions 3,4,6, and 7 and arrive at
$(++\cancel{-}\cancel{+}-\cancel{-}\cancel{+}--)$.

Now, we define a second crystal structure on $\Lambda$ using the usual rule
for tensor products of crystals $\Lambda=\Lambda_1\otimes \cdots
\otimes \Lambda_n$; we call this the {\bf combinatorial} crystal
structure on $\Lambda$.  When we write a $\te_i\bla$ or $\tf_i\bla$,
we will always mean the action in the combinatorial structure, and
always write $\te_iL(\bla)$ or $\tf_iL(\bla)$ for the categorical structure.

Let us describe the tensor product crystal rule in our
language. Assume that the $j$-th group from the left contains the
rightmost uncrossed $+$, and the $j'$-th group contains the leftmost
uncrossed $-$; then we
define \[\tilde{e}_i(\lambda_1,\ldots,\lambda_n)=(\lambda_1,\ldots,
{}_j\tilde{e}_i\lambda_{j'},\ldots,\lambda_n)\qquad
\tilde{f}_i(\lambda_1,\ldots,\lambda_n)=(\lambda_1,\ldots,{}_{j'}\tilde{f}_i\lambda_j,\ldots,\lambda_n).\]
If
there is no uncrossed $-$, we set $\tilde{e}_i (\bla)=0$, and if
there is no uncrossed $+$, we set $\tilde{f}_i\bla=0$.  In the example
above, $j=2$ and $j'=1$.

We remark that the $i$-signature of $\tilde{f}_i\lambda$ is obtained
from the $i$-signature of $\lambda$ by switching the rightmost
uncrossed $+$ to a $-$; the $i$-signature of $\tilde{e}_i\lambda$ is
obtained by switching the leftmost uncrossed $-$ to a $+$.
\excise{Also note that the numbers $k,k'$ can be described
  alternatively as follows.  Take the $i$-signature of $\lambda$ and
  start to transform it using the following rules. On the first step
  we take a consecutive pair $-+$ and replace it with $00$. On the
  next steps we take a consecutive sequence $-0\ldots0+$ (we may have
  no zeroes) and replace it with $0\ldots0$ preserving the length. We
  finish when no $+$ occurs to the right of a $-$.  The resulting
  sequence is called the reduced $i$-signature of $\lambda$. Clearly,
  $k$ is the index of the rightmost $+$ and $k'$ is the index of the
  leftmost $-$. The number of $\pm$ in the reduced signature is
  denoted by $h_{\pm}(\lambda)$.  Clearly, $h_-(\lambda)$ is the
  maximal number $\ell$ such that $\tilde{e}_i^\ell \lambda\neq 0$,
  and a similar description in terms of $\tilde{f}_i$ is true for
  $h_+(\lambda)$.}

Here is our  main theorem concerning the crystal structure of $\Cat$. It is a partial generalization
of the main result of \cite{LoHWCI} and answers a question of the
second author from \cite[\ref{m-sec:crystal}]{Webmerged}.

\begin{Thm}\label{Thm:cryst}
The categorical and combinatorial crystal structures coincide.  That
is \[\tilde{e}_i L(\lambda)=L(\tilde{e}_i\lambda)\quad \text{ and
}\quad \tilde{f}_i L(\lambda)=L(\tilde{f}_i\lambda).\]
\end{Thm}

Since this can be checked one simple root at a time, we fix $i$ and
suppress the subscript $i$ in $E_i,\tilde{e}_i,\alpha_i$,
etc. throughout the rest of this section.

There is a weaker version of this claim that is easy to see. For $\bla=(\lambda_1,\ldots,\lambda_n)$
we write ${}_{j}\tilde{e} \bla$ for $(\lambda_1,\ldots,{}_{j}\tilde{e}\lambda_j,\ldots,\lambda_n)$. The notation
${}_{j}\tilde{f}\bla$ has a similar meaning.
\begin{Lem}\label{Lem:cryst_easy}
We have $\tilde{e} L(\bla)=L({}_{j}\tilde{e}\bla), \tilde{f} L(\bla)=L(\,_{j'}\!\tilde{f}\bla)$
for some $j,j'$.
\end{Lem}
\begin{proof}
We remark that $E \oDelta(\bla)\twoheadrightarrow E L(\bla)$ and so $\tilde{e} L(\bla)$
lies in the head of $E\oDelta(\bla)$. Thanks to (HWC3), the head of $E\oDelta(\bla)$ lies the direct
sum of the heads of the modules $\Delta_{\varrho(\bla)+\alpha^j}({}_{j}\tilde{e} L_{\varrho(\bla)}(\bla))$.
Since ${}_{j}\tilde{e} L_{\varrho(\bla)}(\bla)$ has simple head, equal to $L_{\varrho(\bla)+\alpha^j}({}_{j}\tilde{e}\bla)$,
we see that the head of $\Delta_{\varrho(\bla)+\alpha^j}({}_{j}\tilde{e} L_{\varrho(\bla)}(\bla))$ equals
$L({}_{j}\tilde{e}\bla)$. This completes the proof of the first equality of the lemma. The proof of the second one
is completely analogous.
\end{proof}

\subsection{Ext vanishing}
As in \cite{LoHWCI}, our proof of Theorem \ref{Thm:cryst} is based on vanishing of certain Ext's between $\Delta(\bmu)$'s
and $L(\bla)$'s. A result here is a direct generalization of \cite[Proposition 5.5]{LoHWCI}.

For $k=1,\ldots,n$
and a label $\bla'$, define an integer $h_{-,k}(\bla')$ as the number of $-$'s
in the reduced form of the part of the $i$-signature of $\bla'$ lying in the groups to the
right of and including the $k$th. In the example
we have used before, $h_{-,1}(\bla')=h_{-,2}(\bla')=3, h_{-,3}(\bla')=2$. In general,
$h_{-,1}(\bla'),\ldots, h_{-,n}(\bla')$ is a (non-strictly) decreasing sequence.

Consider integers $m>0$ such that $E^m
L(\bla)=0$ and $\ell\in \{1,\ldots,n\}$ with $h_{-,\ell}(\bla)< m$.
We can always take $\ell=n$, since if
$h_{-,n}(\bla)\geqslant m$, then $E^m L(\bla)\neq 0$. Indeed, $L({}_n\tilde{e}^m\bla)$ appears
as a composition factor in $E^m L(\bla)$,  which one can prove similarly to checking $a_i=b_i$ in the proof of Theorem \ref{main2}.

\begin{Prop}\label{Prop:ext_vanish}
Assume that $\tilde{e}L(\bla')=L(\tilde{e}\bla')$ holds whenever $E^{m-1}L(\bla')=0$ and $|\varrho(\bla')|-|\varrho(\bla)|$
is a multiple of $\alpha_i$. Then, in the above notation, we
have \[\Ext^r(\Delta(\bmu),L(\bla))=0\quad \text{for} \quad r\leqslant h_{-,\ell}(\bmu)-m.\]
\end{Prop}
Our proof closely follows that of \cite[Proposition 5.5]{LoHWCI}. We will need an analog of \cite[Lemma 5.6]{LoHWCI}.
To state it, we need some more notation. Pick a label $\bmu$. Let $\Lambda_j(\bmu)$
denote the set of labels in the labeling set $\Lambda$ for $\Cat$ consisting of all $\bmu'=(\mu'_1,\ldots,\mu'_\ell)$
such that the projective $P_{\Cat^j}(\mu'_j)$ appears in ${}_jF P_{\Cat^j}(\mu_j)$ and $\mu'_p=\mu_p$ for $p\neq j$.
\begin{Lem}\label{Lem:h-}
Assume that a label $\bmu$ and an integer $\ell$  satisfy $h_{-,\ell}(\bmu)>h_{-,\ell+1}(\bmu)$. Set $\bar{\bmu}:=\,_\ell\!\tilde{f}\bmu$.
For any  $\bmu'\in \Lambda_j(\bar{\bmu})$, the following holds:
\begin{enumerate}
\item If $j<\ell$, then $h_{-,\ell}(\bmu')=h_{-,\ell}(\bar{\bmu})=h_{-,\ell}(\bmu)-1$.
\item If $j=\ell$ and $\bmu'\neq\bmu$, then $h_{-,\ell}(\bmu')>h_{-,\ell}(\bmu)$.
\item If $j> \ell$, then $h_{-,\ell}(\bmu')>h_{-,\ell}(\bmu)$.
\end{enumerate}
\end{Lem}
\begin{proof}
The first claim follows from $\bmu'_p=\bar{\bmu}_p$ for $p\geqslant \ell$. To prove the second claim we note that
if  $P_{\Cat^j}(\mu'_j)$ appears in ${}_jF P_{\Cat^j}(\mu_j)$, then $h_-(L_{\Cat^j}(\mu'_j))\geqslant h_-(L_{\Cat^j}(\mu_j))+1$
with the equality if and only if $\mu'_j={}_{j}\tilde{f} \mu_j$. This follows from \cite[Proposition 5.20]{CR04} (applied
to ${}_jE$ using $\Hom({}_jF P_{\Cat^j}(\mu_j'), L_{\Cat^j}(\xi))=\Hom(P_{\Cat^j}(\mu_j'), {}_jE L_{\Cat^j}(\xi))$).
Now the claims of (2) and (3) become purely combinatorial and have
proofs like those of \cite[Lemma 5.6]{LoHWCI}.
\end{proof}

\begin{proof}[Proof of Proposition \ref{Prop:ext_vanish}]
We may assume that $\ell=n$ or $h_{-,\ell}(\bmu)>h_{-,\ell+1}(\bmu)$, since
otherwise the desired equality is a special case of the same claim for $\ell+1$.
In particular,  $\bar{\bmu}:=\,_\ell\!\tilde{f}\bmu$ is not 0. Fix $r\leqslant h_{-,\ell}(\bmu)-m$.
We are going to prove $\Ext^r(\Delta(\bmu), L(\bla))=0$. Below all elements $\bla'$ we consider
satisfy the condition that $|\varrho(\bla')|-|\varrho(\bla)|$
is a multiple of $\alpha_i$. For all $\bmu'$ we consider, we assume the difference of the weights of $\mu_j',\mu_j$
is a multiple of $\alpha_i$ for all $j$.

We can make the following assumptions in the proof:
\begin{itemize}
\item[(i)] For any $\bla'$ with $E^{m-1}L(\bla')=0, $ and any $\bmu'$, we have $\Ext^{r'}(\Delta(\bmu'), L(\bla'))=0$ for $r'\leqslant h_-(\bmu')-m$.
\item[(ii)] For 
any $\bmu'$ with $h^\ell_{-}(\mu'_\ell)<
h^\ell_{-}(\mu_\ell), h^j_-(\mu'_j)\geqslant h^j_-(\mu_j)$ for all $j>\ell$ with at least one strict inequality, we have $\Ext^{r'}(\Delta(\bmu'),L(\bla))=0$
for any $r'\leqslant h_{-,\ell}(\bmu')-m$.
\item[(iii)] For any $r'<r$, 
$\bmu'$ with $\mu'_j=\mu_j$ for $j>\ell$, $h_{-,\ell}(\bmu')\geqslant m+r'$, we have $\Ext^{r'}(\Delta(\bmu'), L(\bla))=0$.
\end{itemize}

If (i) does not hold, we can replace $m$ with $m-1$, note that we can take $\ell=1$ thanks to the assumption on $m$
in the proposition. If (ii) or (iii) do not hold, we replace $(\bmu,r)$ with $(\bmu',r')$. This procedure terminates.
Indeed, the sum $\sum_{j>\ell}h^j_-(\mu_j)$ weakly increases on every step (and strictly increases if we do a replacement as in (ii)). On the other hand, the sum cannot increase indefinitely, this follows from our restriction on $\bmu'$
and the fact that all tensor factors $V(\nu_j)$ are highest weight and integrable. So we can apply a replacement described
in (ii) only finitely many times. On the other hand, in (iii), we decrease $r$. This establishes the termination.



Consider the object $F \Delta(\bar{\bmu})$ and its filtration described in (TPC3). Let $\mathcal{F}$
denote the filtration subobject with successive filtration quotients $\Delta({}_{j}F P_{\varrho(\bar{\bmu})}(\bar{\bmu})), j\leqslant \ell$.
The object $\mathcal{F}$ has a quotient isomorphic to
$h_+(\bar{\mu}_\ell)$ copies of $\Delta(\bmu)$.
Let $\mathcal{F}_0$ denote the kernel of this quotient. Then all standards $\Delta(\bmu')$ occurring in the filtration of
$\mathcal{F}_0$ satisfy (1) or (2) of Lemma \ref{Lem:h-} and $\bmu'$ is as in (iii) above. Similarly, all standards
$\Delta(\bmu')$ in the filtration of
$F\Delta(\bar{\bmu})/\mathcal{F}$ satisfy (3) of Lemma \ref{Lem:h-}, and $\bmu'$ is as in (ii).

It follows  from (ii) that \begin{equation}\label{eq:Ext_vanish1}\Ext^{r+1}(F\Delta(\bar{\bmu})/\mathcal{F},L(\bla))=0.\end{equation}
It follows from (iii)  that
\begin{equation}\label{eq:Ext_vanish2}\Ext^{r-1}(\mathcal{F}_0,L(\bla))=0.\end{equation}

We claim that \begin{equation}\label{eq:Ext_vanish3}\Ext^r(F
  \Delta(\bar{\bmu}),L(\bla))=0.\end{equation}
Indeed, by the biadjointness, $\Ext^r(F \Delta(\bar{\bmu}),L(\bla))=
\Ext^r(\Delta(\bar{\bmu}),E L(\bla))$. All simple constituents $L(\bla')$ of $E L(\bla)$ satisfy the assumptions
of (i). Since $h_{-,\ell}(\bar{\bmu})=h_{-,\ell}(\bmu)-1$ and $r\leqslant (h_{-,\ell}(\bmu)-1)-(m-1)=h-m$, we deduce (\ref{eq:Ext_vanish3})
from (i).

Using standard short exact sequences for $\Ext$'s (compare with the proof of \cite[Proposition 5.5]{LoHWCI})
together with (\ref{eq:Ext_vanish1}),(\ref{eq:Ext_vanish2}),(\ref{eq:Ext_vanish3}), we see that
$\Ext^r(\mathcal{F}/\mathcal{F}_0, L(\bla))=0$. Since $\mathcal{F}/\mathcal{F}_0$ is the direct sum of several copies of
$\Delta(\bmu)$, we are done.
\end{proof}

\subsection{Proof of the main theorem}
Our proof basically repeats that in \cite{LoHWCI}.  Fix a weight $\mu$ for $\g$.  Of course, we can restrict our
 attention to  the weight subcategories with weights of the form
  $\mu+r\alpha_i, r\in \Z$.

Now, let us prove the claim that our crystal operators agree by induction on
  $w(\bla)=\langle\varrho(\bla),\alpha_i^\vee\rangle$. The set of
  values of $w(\bla)$ is finite since our representations are
  integrable.  In order to organize our induction, we consider two statements:
  \begin{itemize}
  \item [$(\phi_w)$] For all
  $L(\bla)$ with $w(\bla)\geqslant w$, we have that $\tilde{e}L(\bla)=L(\tilde{e}\bla)$.
  \item[$(\eta_w)$] We have both $(\phi_{w+1})$ and that $h_-(\bla)=h_-(L(\bla))$ for all $\bla$ with
    $w(\bla)\geqslant w$.
  \end{itemize}
Obviously, proving $(\phi_w)$ for all $w$ will complete the proof of
Theorem \ref{Thm:cryst}; we
will proceed in establishing $(\phi_{w+1})\Rightarrow (\eta_w)
\Rightarrow (\phi_w)$.

The base of induction is the statement $(\phi_w)$ where we take $w$ to be
  the minimal amongst those realized.  In this case, both sides of the
  equality are 0.

\begin{proof}[Proof that $(\phi_{w+1})\Rightarrow (\eta_w)$]

First, we establish the inequality $h_-(L(\bla))\geqslant h_-(\bla)$. Assume the contrary, $h_-(L(\bla))<h_-(\bla)$.
Then  $h_{-}(\bla)>0$ and so
  the element $\bar{\bla}:=\tilde{e}\bla$ is nonzero. Let $j\in
  \{1,\ldots,n\}$ be the unique index such that $\bar{\bla}={}_j\tilde{e}\bla$. Consider the object
  $F\Delta(\bar{\bla})$. Form the filtered subobjects $\mathcal{F}_0,\mathcal{F}$
  as in the proof of Proposition \ref{Prop:ext_vanish} so that
  $\mathcal{F}/\mathcal{F}_0$ is the direct sum of several copies of
  $\Delta(\bla)$.

Pick $\Delta(\bmu)$ appearing in
  $F\Delta(\bar{\bla})/\mathcal{F}$ and set $m=h_-(\bla), \ell=j$ so
  that $h_{-}(\bla)=h_{-,\ell}(\bla)$.   By (3) of Lemma \ref{Lem:h-},
  $h_{-,\ell}(\bmu)>m$. The assumptions of Proposition
  \ref{Prop:ext_vanish} are satisfied. It follows that
  $\Ext^1(\Delta(\bmu),L(\bla))=0$.
Thus,  $L(\bla)$ lies in the head of
  $F\Delta(\bar{\bla})$.

By \cite[Lemma 5.11]{CR04}, this implies that
  $h_-(L(\bla))\geqslant h_-(L(\bar{\bla}))+1$. By the inductive
  assumption, \[h_-(L(\bar{\bla}))=h_-(\bar{\bla})=h_-(\bla)-1.\] So
  we see that $h_-(L(\bla))\geqslant h_-(\bla)$, our desired
  inequality in
  the weight subcategories with $w(\bla)=w$.

Note that the
  crystal of $\Cat$ is isomorphic to the crystal of $V_1\otimes
  V_2\otimes\ldots\otimes V_n$ by \cite[5.55]{BeKa}
  combined with \cite[Proposition 5.20]{CR04}.
  Our crystal on $\Lambda$ is also isomorphic to that of $V_1\otimes
  V_2\otimes\ldots\otimes V_n$, since it is the standard tensor
  product crystal structure.  In particular, we see that for any
  $w,h$ the number of $\bla$ with $w(\bla)=w$ and $h_-(L(\bla))=h$
  equals to the number of $\bla$ with $w(\bla)=w$ and
  $h_-(\bla)=h$. So the inequality $h_-(L(\bla))\geqslant h_-(\bla)$
  must be an equality for all $\bla$ with
  $w(\bla)=w$; that is, the implication $(\phi_{w+1})\Rightarrow(\eta_w)$ is now proved.
\end{proof}
\begin{proof}
[Proof that $(\eta_w)\Rightarrow
 (\phi_w)$]  It suffices to prove instead that $\tilde{f}L(\bla')=L(\tilde{f}\bla')$
  for all $\bla'$ with $w(\bla')=w-2$.

  First of all, we have that $\tilde{f}L(\bla)=0$ if and only if $\tilde{f}\bla=0$. Indeed, since $h_-(L(\bla))=h_-(\bla)$, we know that
  $h_+(L(\bla))=h_+(\bla)$. So we may assume that $\tilde{f}L(\bla)\neq 0\neq \tilde{f}\bla$.

  We say that $\bla$ and $\bla'$ lie in the same {\bf $i$-family} if one can
  obtain $\bla$ from $\bla'$ by applying maps ${}_{j}\tilde{e}_i$ and
  ${}_{j}\tilde{f}_i$. We will order elements of the family using
  reverse lexicographic (i.e. reading from the right) order $\succ$ on the
  $i$-signatures with $+> -$.
  In our proof we may assume that
  $\tilde{f}L(\bla')=L(\tilde{e}\bla')$ is proved for all $\bla'$ such
  that $h_-(\bla')=h_-(\bla)$ and $\tilde{f}\bla\succ\tilde{f}\bla'$.

  Now, fix $\bla$ and define  $\tilde{\bla}$ by
  $\tilde{f}L(\bla)=L(\tilde{\bla})$. We already know that
  \[h_-(\tilde{f}\bla)=h_-(\tilde{\bla})=h_-(\bla)+1.\] By
  Lemma \ref{Lem:cryst_easy}, we know that $\tilde{f}\bla$ is amongst the labels
  ${}_{j}\tilde{f}\bla$ with $h_-({}_{j}\tilde{f}\bla)=h_-(\bla)+1$.

  Assume that $\tilde{\bla}\neq \tilde{f}\bla$.  Since  $h_-(\tilde{\bla})-1=h_-(\bla)\geqslant 0$, we see that
  $\bla':=\tilde{e}\tilde{\bla}\neq 0$ and so $\tilde{\bla}=\tilde{f}
  \bla'$.  If  $\tilde{f}\bla\succ \tilde{\bla}=\tilde{f}
  \bla'$, then by the inductive
  assumption \[\tilde{f}L(\bla')=L(\tilde{f}\bla')=L(\tilde{\bla})=\tilde{f}L(\bla)\]
  which is impossible.

Thus, we must have  $\tilde{\bla}\succ \tilde{f}\bla$.
But then Lemma \ref{Lem:h-}
  implies $h_-(\tilde{\bla})> h_-(\tilde{f} \bla)$, which is another
  contradiction. This proves $\tilde{f}L(\bla)=L(\tilde{f}\bla)$, and
  thus implies $(\phi_w)$.
  This closes the circle of the induction and completes the proof.
\end{proof}

%

\bibliography{./gen}

\def\cprime{$'$}
\providecommand{\bysame}{\leavevmode\hbox to3em{\hrulefill}\thinspace}
\providecommand{\MR}{\relax\ifhmode\unskip\space\fi MR }
\providecommand{\MRhref}[2]{%
  \href{http://www.ams.org/mathscinet-getitem?mr=#1}{#2}
}
\providecommand{\href}[2]{#2}
\begin{thebibliography}{Web17b}

\bibitem[BK07]{BeKa}
Arkady Berenstein and David Kazhdan, \emph{Geometric and unipotent crystals.
  {II}. {F}rom unipotent bicrystals to crystal bases}, Quantum groups, Contemp.
  Math., vol. 433, Amer. Math. Soc., Providence, RI, 2007, pp.~13--88.

\bibitem[BLW17]{BLW}
Jonathan Brundan, Ivan Losev, and Ben Webster, \emph{Tensor product
  categorifications and the super {K}azhdan-{L}usztig conjecture}, Int. Math.
  Res. Not. IMRN (2017), no.~20, 6329--6410. \MR{3712200}

\bibitem[CL15]{CaLa}
Sabin Cautis and Aaron~D. Lauda, \emph{Implicit structure in 2-representations
  of quantum groups}, Selecta Math. (N.S.) \textbf{21} (2015), no.~1, 201--244.
  \MR{3300416}

\bibitem[CPS96]{CPS96}
Edward Cline, Brian Parshall, and Leonard Scott, \emph{Stratifying endomorphism
  algebras}, Mem. Amer. Math. Soc. \textbf{124} (1996), no.~591, viii+119.

\bibitem[CR08]{CR04}
Joseph Chuang and Rapha{\"e}l Rouquier, \emph{Derived equivalences for
  symmetric groups and {$\mathfrak {sl}_2$}-categorification}, Ann. of Math.
  (2) \textbf{167} (2008), no.~1, 245--298.

\bibitem[HY13]{HY}
Jiuzu Hong and Oded Yacobi, \emph{Polynomial representations and
  categorifications of {F}ock space}, Algebr. Represent. Theory \textbf{16}
  (2013), no.~5, 1273--1311. \MR{3102954}

\bibitem[KL09]{KLI}
Mikhail Khovanov and Aaron~D. Lauda, \emph{A diagrammatic approach to
  categorification of quantum groups. {I}}, Represent. Theory \textbf{13}
  (2009), 309--347. \MR{2525917 (2010i:17023)}

\bibitem[KL10]{KLIII}
\bysame, \emph{A categorification of quantum {${\rm sl}(n)$}}, Quantum Topol.
  \textbf{1} (2010), no.~1, 1--92. \MR{2628852 (2011g:17028)}

\bibitem[Lau10]{Lausl21}
Aaron~D. Lauda, \emph{A categorification of quantum {${\rm sl}(2)$}}, Adv.
  Math. \textbf{225} (2010), no.~6, 3327--3424. \MR{2729010 (2012b:17036)}

\bibitem[Los13]{LoHWCI}
Ivan Losev, \emph{Highest weight {$\mathfrak{sl}_2$}-categorifications {I}:
  crystals}, Math. Z. \textbf{274} (2013), no.~3-4, 1231--1247. \MR{3078265}

\bibitem[Los15]{LoHWCII}
\bysame, \emph{Highest weight {$\mathfrak{sl}_2$}-categorifications {II}:
  {S}tructure theory}, Trans. Amer. Math. Soc. \textbf{367} (2015), no.~12,
  8383--8419. \MR{3403059}

\bibitem[LV11]{LV}
Aaron~D. Lauda and Monica Vazirani, \emph{Crystals from categorified quantum
  groups}, Adv. Math. \textbf{228} (2011), no.~2, 803--861. \MR{2822211}

\bibitem[Rou]{Rou2KM}
Raphael Rouquier, \emph{2-{K}ac-{M}oody algebras}, \arxiv{0812.5023}.

\bibitem[Rou12]{RouQH}
Rapha{\"e}l Rouquier, \emph{Quiver {H}ecke algebras and 2-{L}ie algebras},
  Algebra Colloq. \textbf{19} (2012), no.~2, 359--410. \MR{2908731}

\bibitem[Sha11]{ShanCrystal}
Peng Shan, \emph{Crystals of {F}ock spaces and cyclotomic rational double
  affine {H}ecke algebras}, Ann. Sci. \'Ec. Norm. Sup\'er. (4) \textbf{44}
  (2011), no.~1, 147--182. \MR{2760196 (2012c:20009)}

\bibitem[Soe90]{Soe90}
Wolfgang Soergel, \emph{Kategorie {$\mathcal{ O}$}, perverse {G}arben und
  {M}oduln \"uber den {K}oinvarianten zur {W}eylgruppe}, J. Amer. Math. Soc.
  \textbf{3} (1990), no.~2, 421--445. \MR{MR1029692 (91e:17007)}

\bibitem[Web17a]{Webmerged}
Ben Webster, \emph{Knot invariants and higher representation theory}, Memoirs
  of the American Mathematical Society \textbf{250} (2017), no.~1191, 141.

\bibitem[Web17b]{WebRou}
\bysame, \emph{Rouquier's conjecture and diagrammatic algebra}, Forum of
  Mathematics. Sigma \textbf{5} (2017), e27, 71. \MR{3732238}

\end{thebibliography}
\bibliographystyle{amsalpha}

\end{document}